\documentclass[12pt]{amsart}

\usepackage{amsfonts,amsthm,amsmath,amssymb,amscd,mathrsfs}
\usepackage{graphics}
\usepackage{indentfirst}
\usepackage{cite}
\usepackage{latexsym}
\usepackage[dvips]{epsfig}
\usepackage[utf8]{inputenc}
\usepackage{color}
\usepackage{esint}
\usepackage[colorlinks=true,linkcolor=blue]{hyperref}
\usepackage{verbatim}

\usepackage{bm}

\newtheorem{theorem}{Theorem}
\newtheorem{lemma}[theorem]{Lemma}
\newtheorem{corollary}[theorem]{Corollary}
\newtheorem{proposition}[theorem]{Proposition}

\theoremstyle{definition}
\newtheorem{definition}{Definition}
\newtheorem{assumption}{Assumption}
\newtheorem{example}{Example}

\theoremstyle{remark}
\newtheorem{remark}{Remark}

\def\be{\begin{eqnarray}}
\def\ee{\end{eqnarray}}
\def\ba{\begin{aligned}}
\def\ea{\end{aligned}}
\def\bay{\begin{array}}
\def\eay{\end{array}}
\def\bca{\begin{cases}}
\def\eca{\end{cases}}
\def\p{\partial}
\def\no{\nonumber}
\def\e{\epsilon}
\def\de{\delta}

\def\Om{\Omega}
\def\f{\frac}
\def\th{\theta}
\def\la{\lambda}
\def\lab{\label}
\def\b{\bigg}
\def\var{\varphi}

\def\ka{\kappa}
\def\al{\alpha}

\def\ga{\gamma}
\def\Ga{\Gamma}
\def\ti{\tilde}

\def\ol{\overline}

\def\si{\sigma}

\def\q{\quad}

\def\bt{\begin{theorem}}
\def\et{\end{theorem}}
\def\bc{\begin{corollary}}
\def\ec{\end{corollary}}
\def\bl{\begin{lemma}}
\def\el{\end{lemma}}
\def\bp{\begin{proposition}}
\def\ep{\end{proposition}}
\def\br{\begin{remark}}
\def\er{\end{remark}}
\def\bd{\begin{definition}}
\def\ed{\end{definition}}
\def\bpf{\begin{proof}}
\def\epf{\end{proof}}
\def\bex{\begin{example}}
\def\eex{\end{example}}
\def\bq{\begin{question}}
\def\eq{\end{question}}
\def\bas{\begin{assumption}}
\def\eas{\end{assumption}}
\def\ber{\begin{exercise}}
\def\eer{\end{exercise}}
\def\mb{\mathbb}
\def\mbR{\mb{R}}

\def\mc{\mathcal}

\def\mfd{\mathfrak{d}}
\def\mfe{\mathfrak{e}}
\def\mcD{\mc{D}}
\def\mcF{\mc{F}}
\def\mcG{\mc{G}}
\def\mcK{\mc{K}}
\def\mcS{\mc{S}}
\def \bfe{\mathbf{e}}
\def\bu{\mathbf{u}}
\def\bW{\mathbf{W}}
\def\bx{\mathbf{x}}
\def\bPhi{\boldsymbol{\Phi}}
\def\bPsi{\boldsymbol{\Psi}}

\def\msR{\mathscr{R}}

\begin{document}
\title[Stability of transonic shock]{Structural stability of the transonic shock problem in a divergent three dimensional axisymmetric perturbed nozzle}
\author{Shangkun Weng}
\address{School of Mathematics and Statistics, Wuhan University, Wuhan, Hubei Province, 430072, People's Republic of China.}
 \email{skweng@whu.edu.cn}\author{Chunjing Xie}
 \address{School of Mathematical Sciences, Institute of Natural Sciences, Ministry of Education Key Laboratory
of Scientific and Engineering Computing, Shanghai Jiao Tong University, Shanghai 200240,
 China.}
  \email{cjxie@sjtu.edu.cn}
  \author{Zhouping Xin}
  \address{The Institute of Mathematical Sciences and department of mathematics, The Chinese University of Hong Kong, Shatin, NT, Hong Kong.}
  \email{zpxin@ims.cuhk.edu.hk}

\begin{abstract}
  In this paper, we prove the structural stability of the transonic shocks for three dimensional axisymmetric Euler system with swirl velocity under the perturbations for the incoming supersonic flow, the nozzle boundary, and the exit pressure. Compared with the known results on the stability of transonic shocks, one of the major difficulties for the axisymmetric flows with swirls is that corner singularities near the intersection point of the shock surface and nozzle boundary and the artificial singularity near the axis appear simultaneously.  One of the key points in the analysis for this paper is the introduction of an {\it invertible} Lagrangian transformation which can straighten the streamlines in the whole nozzle and help to represent the solutions of transport equations explicitly.
\end{abstract}

\date{}
\maketitle

\section{Introduction and main results}
The three-dimensional steady inviscid gas motion is governed by the following compressible Euler system
\be\label{com-euler}
\begin{cases}
\text{div }(\rho {\bf u})=0,\\
\text{div }(\rho {\bf u}\otimes {\bf u}+ P I_n) =0,\\
\text{div }(\rho (\f12|{\bf u}|^2 +\mfe) {\bf u}+ P{\bf u}) =0,
\end{cases}
\ee
where ${\bf u}=(u_1, u_2,u_3)$, $\rho$, $P$, and $\mfe$ stand for the velocity, density, pressure, and internal energy, respectively. Suppose that the gas is polytropic. Then the equation of state and the internal energy are of the form
\begin{equation}\label{eqstate}
P=A \rho^{\ga} e^{\f{S}{c_v}}\quad \text{and}\quad
\mfe=\f{P}{(\ga-1)\rho},
\end{equation}
 respectively, where $\gamma\geq 1$, $A$, and $c_v$ are positive constants, and $S$ is called the specific entropy.
 The system \eqref{com-euler} is a hyperbolic system for supersonic flows ($M_a>1$), a hyperbolic-elliptic  coupled  system for subsonic flows ($M_a<1$), and degenerate at sonic point (i.e. $M_a=1$), respectively, where $M_a=\f{|{\bf u}|}{c(\rho,S)}$ is called the Mach number of the flows with $c(\rho,S)=\sqrt{\partial_\rho P(\rho, S)}$ called the local sound speed.

In this paper, we are interested in the basic transonic shock problem in a De Laval nozzle described by Courant and Friedrichs (\cite[Page 386]{cf48}): given appropriately large receiver pressure $P_e$, if the upstream flow is still supersonic behind the throat of the nozzle, then at a certain place in the diverging part of the nozzle a shock front intervenes and the gas is compressed and slowed down to subsonic speed. The position and the strength of the shock front are automatically adjusted so that the end pressure at the exit becomes $P_e$. The stability of transonic shocks in nozzles is a fundamental problem in gas dynamics that have been studied extensively  in various situations. The early studies for transonic flows, in particular for quasi-one dimensional models,  can be found in \cite{bers58,liu82,egm84}.  The structural stability of  transonic shocks  for multidimensional steady potential flows in  nozzles was studied in \cite{ChenF, xy05, xy08a}. It was showed in \cite{xy05,xy08a} that the stability of transonic shock for potential flows is usually ill-posed under the perturbation of the exit pressure. Later on, it was proved that the transonic shock problem in the flat nozzle with small perturbations is either ill-posed  under general perturbations of the exit pressure or well-posedness if the exit pressure satisfies a special constraint, see \cite{chen05,chen08,cy08,ly08,lxy10b} and the references therein.
There have been many interesting results on transonic shock problems in a nozzle for different models with various exit boundary conditions, for example,  the non-isentropic potential model, the exit boundary condition for the normal velocity, the spherical flows without boundary, etc, see \cite{bf11, ccs06, ccf07, lxy16} and references therein. The well-posedness of the transonic shock problem was first established in a special class of two dimensional divergent nozzle under the perturbations for the exit pressure in \cite{lxy09a}. Later on, the results were generalized to the problem in general two dimensional divergent nozzles, see \cite{lxy09b,lxy13}. In particular, in \cite{lxy13}, the Courant-Friedrich's transonic shock in a two dimensional straight divergent nozzle is shown to be structurally stable under generic perturbations for both the nozzle shape and the exit pressure, and optimal regularity of solutions are also obtained. Such a structural stability also holds for perturbations of incoming supersonic flows \cite{Weng}. The key idea there is to introduce a Lagrangian transformation to straighten the streamlines and reduce the Euler system with the shock to a second order elliptic equation with a nonlocal term  and an unknown parameter together with an ODE for the shock front. In \cite{lxy10a,lxy10b}, the existence and stability of  transonic shock for three dimensional axisymmetric flows without swirl in a conic nozzle was proved to be structurally stable under suitable perturbations of the exit pressure.

In this paper, we study the stability of transonic shocks for 3D axisymmetric flows with swirls under the perturbations of the exit pressure, the nozzle wall, and supersonic incoming flows. First, let us introduce the standard spherical coordinates
\begin{equation*}
\left\{
\begin{aligned}
&x_1 =r\cos\theta,\\
& x_2=r\sin\theta\cos\varphi,\\
&x_3=r\sin\theta\sin\varphi
\end{aligned}
\right.
\quad\text{and}\quad
\left\{
\begin{aligned}
&{\bf e}_r =(\cos\th,\sin\theta \cos\var,\sin\theta \sin\var)^t,\\
&{\bf e}_{\th}=(-\sin\th,\cos\theta \cos\var,\cos\th \sin\var)^t,\\
& {\bf e}_{\var}=(0,-\sin\var,\cos\var)^t.
\end{aligned}
\right.
\end{equation*}
Let ${\bf u}=U_1 {\bf e}_r+ U_2{\bf e}_{\th} + U_3{\bf e}_{\var}$.
The three dimensional axisymmetric Euler system can be written as
\begin{equation}\lab{euler-sph1}
\left\{
\begin{aligned}
&\p_r(r^2 \rho U_1 \sin\theta) +\p_{\theta} (r \rho U_2\sin\theta)=0,\\
&\rho U_1\p_r U_1+\f{1}{r}\rho U_2 \p_{\theta} U_1  + \p_r P- \f{\rho (U_2^2+U_3^2)}{r}=0,\\
&\rho U_1\p_r U_2+\f{1}{r}\rho U_2 \p_{\theta} U_2+\f{1}{r}\p_{\theta} P +\f{\rho U_1 U_2}{r}-\f{\rho U_3^2}{r}\cot\theta=0,\\
&\rho U_1\p_r (r U_3\sin\theta)+\f{1}{r}\rho U_2 \p_{\theta} (r U_3\sin\theta) =0,\\
&\rho U_1\p_r S+\f{1}{r}\rho U_2 \p_{\theta} S =0.
\end{aligned}
\right.
\end{equation}
Suppose that $\theta_0\in (0, \frac{\pi}{2})$, $r_1$, $r_2(>r_1)$ are fixed  positive constants.  Let $\Omega_b= \{(r, \theta):  r\in (r_1, r_2), \theta\in [0, \theta_0)\}$ be a straight divergent nozzle and $\Gamma_b=\partial\Omega_b$ be its boundary.
\begin{center}
\includegraphics[height=5cm, width=10cm]{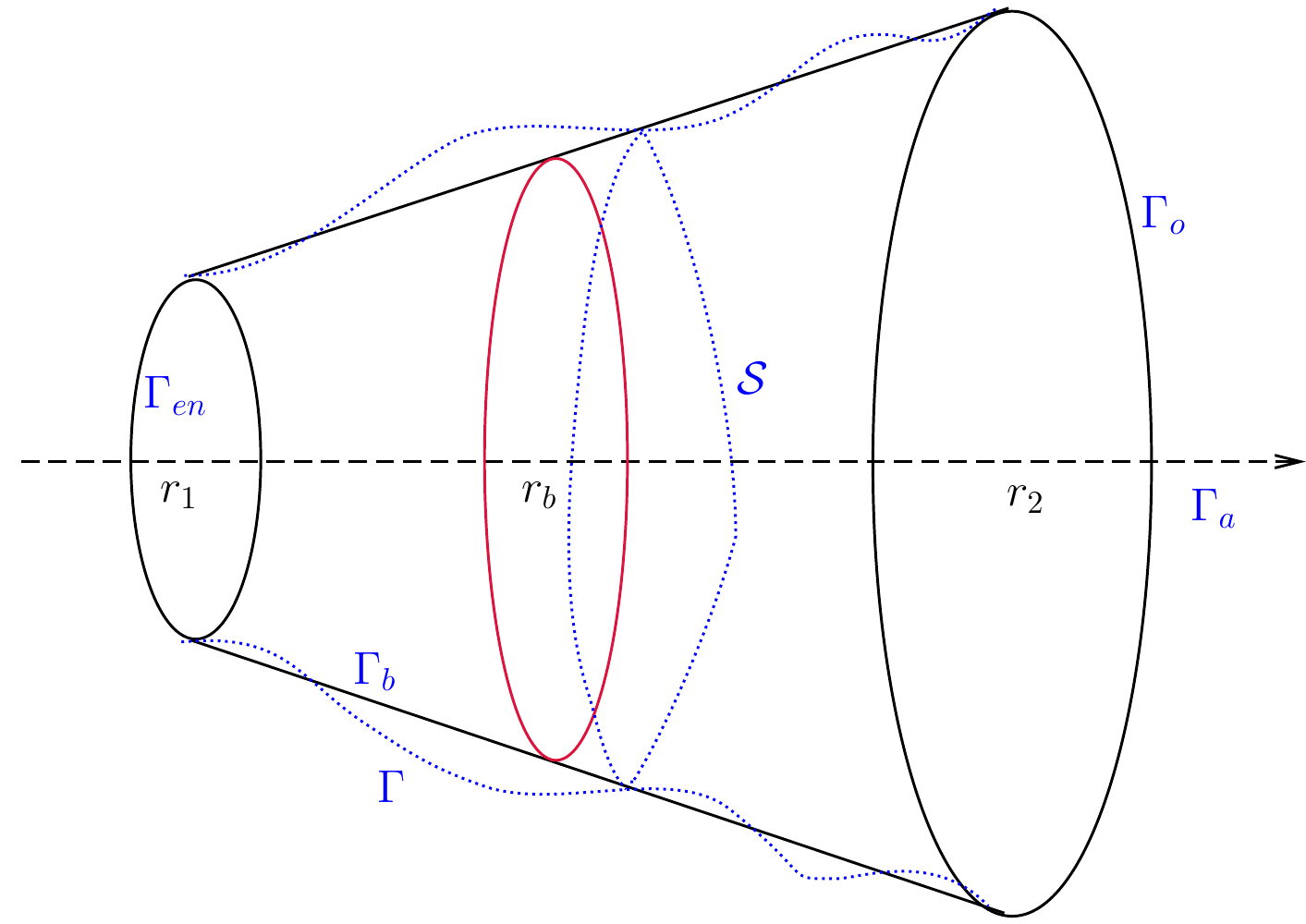}\\
{\small Figure 1.  The straight and perturbed nozzles}
\end{center}

Suppose that the incoming supersonic flow is prescribed at the inlet $r=r_1$, i.e.,
\begin{equation}
{\bu}^-(\bx)=U_b^-(r_1){\bfe}_r,\quad  P_b^-(\bx)=P_b^-(r_1)>0,\quad  S_b^-(\bx)=S_b^-, \quad \text{ at  } r=r_1,
 \end{equation}
where $U_b^-(r_1)>c(\rho_b^-(r_1),S_b^-)>0$ and $S_b^-$ is a constant. There exist two positive constants $P_1$ and $P_2$ which depend only on the incoming supersonic flows and the nozzle, such that if the pressure $P_e\in (P_1, P_2)$ is given at the exit $r=r_2$,  then there exists a unique piecewise smooth spherical symmetric transonic shock solution
\be\lab{background}
\bPsi_b(\bx)= ({\bu}_b,P_b,S_b)(\bx)=\left\{
\begin{aligned}
&{{\bPsi}}_b^-(\bx):=(U_b^{-}(r),0, 0, P_b^{-}(r),S_b^{-}),\,\,\text{in }\Omega_b^-\\
&{{\bPsi}}_b^+(\bx):=(U_b^{+}(r), 0, 0, P_b^{+}(r),S_b^{+}),\,\,\text{in }\Omega_b^+
\end{aligned}
\right.
\ee
to \eqref{com-euler} with a shock front located at $r=r_b\in (r_1,r_2)$, where
\begin{equation}
\Om_{b}^-=\Omega_b\cap \{ r\in (r_1,r_b)\}\quad \text{and}\quad  \Om_{b}^+=\Omega_b\cap \{ r\in (r_b,r_2)\}.
\end{equation}
Across the shock, the Rankine-Hugoniot conditions and the physical entropy condition are satisfied:
\begin{equation}\label{RH_background}
[\rho U_b]\Big|_{r=r_b}=0,\quad [\rho_b U_b^2+P_b]\Big|_{r=r_b}=0,\quad [B]\Big|_{r=r_b}=0, \quad S_b^+>S_b^-,
\end{equation}
where $B=\frac{|\bu|^2}{2}+\mfe+\frac{P}{\rho}$ is called the Bernoulli function and $[g]\Big|_{r=r_b}:=g(r_b+)-g(r_b-)$ denotes the jump of $g$ at $r=r_b$. Later on,  this special solution, $\bPsi_b$, will be called the background solution. Clearly, one can extend the supersonic and subsonic parts of $\bPsi_b$ in a natural way, respectively. With an abuse of notations, we still call the extended subsonic and supersonic solutions $\bPsi_b^+$ and $\bPsi_b^-$, respectively.  One can refer to \cite[Section 147]{cf48} or \cite[Theorem 1.1]{xy08b} for more details of this spherical symmetric transonic shock solution. The main goal of this paper is to establish the structural stability of this spherical symmetric transonic shock solution under axisymmetric perturbations of the incoming supersonic flows, the nozzle walls, and the exit pressure.

The perturbed nozzle is $\Om=\{(r,\theta): r_1<r<r_2, 0\leq \theta \leq \theta_0 + \e f(r)\}$, where $\e$ is a small positive constant and $f\in C^{2,\al}([r_1,r_2])$ satisfies
\be\lab{wall}
f(r_1)=f'(r_1)=0.
\ee
Suppose that the incoming supersonic flow at the inlet $r=r_1$ is given
by \begin{equation}\lab{super1}
\bPsi\Big|_{r=r_1}:=(U_1^-,U_2^-,U_3^-, P^-,S^-)\Big|_{r=r_1}=\bm{\Psi}_{en}^-=\bm{\Psi}_b^-+ \e \bm{\Psi}_p(\theta),
\end{equation}
where
\begin{equation}\lab{super2}
\bm{\Psi}_p(\theta)=(U_{1,p}^-,U_{2,p}^-,U_{3,p}^-, P_p^-,S_p^-)(\theta)\in (C^{2,\al}([0,\theta_0]))^5
\end{equation}
The flow satisfies  the slip condition ${\bf u}\cdot {\bf n}$=0 on the nozzle wall, where ${\bf n}$ is the outer normal of the nozzle wall. In terms of spherical coordinates, the slip boundary condition for the axisymmetric flows can be written as
\be\lab{slip1}
U_2 =\e r f'(r) U_1\q \text{on } \Gamma:=\{(r,\theta): \theta= \theta_0+\e f(r),\q r_1\leq r\leq r_2\}.
\ee
At the exit of the nozzle, the end pressure is prescribed by
\be\lab{pressure}
P(x)= P_e + \e P_0(\theta)\ \text{at}\ \ \Gamma_o:=\{(r_2,\theta): \theta \in (0, \theta_0)\},
\ee
here $P_0\in C^{1,\al}([0, 2\theta_0])$ (in fact, what is needed in this paper is that $P_0$ is a $C^{1, \alpha}$ function in a region slightly larger than $[0, \theta_0]$).

Since the steady Euler system for supersonic flow is hyperbolic, if the incoming data satisfies the following compatibility conditions
\be\lab{super3}\begin{cases}
U_{2,p}^-(0)=U_{3,p}^-(0)=\f{d^2}{d\th^2}U_{2,p}^-(0)=\f{d}{d\theta} P_p^-(0)=\f{d}{d\theta} U_{3,p}^-(0)=\frac{d}{d\theta}S_p^-(0)=0,\\
U_{2,p}^-(\theta_0)=0,\ \f{d}{d\theta} P_p^-(\theta_0)= (U_{3,p}^-(\theta_0))^2\cot\theta_0,
\end{cases}\ee
then the problem for the system \eqref{euler-sph1} together with \eqref{super1} and \eqref{slip1} can be solved by
 the characteristic method and Picard iteration (see \cite{john90}). Furthermore, for small $\e>0$, there exists a unique $C^{2,\al}(\ol{\Om})$ solution $\bm\Psi^-=(U_1^-,U_2^-,U_3^-,P^-,S^-)(r,\theta)$ to \eqref{com-euler}, which does not depend on $\var$ and satisfies the following properties
\be\lab{super4}
\|(U_1^-,U_2^-,U_3^-,P^-,S^-)-({U}_{b}^-,0,0,{P}_b^-, {S}_b^-)\|_{C^{2,\al}(\overline{\Om})}\leq C_0\e,
\ee and
\be\lab{super5}
\begin{aligned}
U_2^-=U_3^- =\f{\p}{\p \th}(U_1^-,U_3^-, P^-,S^-)= \f{\p^2}{\p \th^2}U_2^-=0, \,\, \text{at}\,\, \Gamma_a:=\{(r, 0): r_1< r< r_2\}.
\end{aligned}
\ee

Now we are looking for a piecewise smooth solution $\bm\Psi$ for \eqref{euler-sph1} supplemented with the boundary conditions \eqref{super1}, \eqref{slip1}, and \eqref{pressure}, which jumps only at a shock front at $\mc{S}=\{(r, \theta): r=\xi(\theta), 0\leq \theta \leq \theta_0\}$. More precisely, $\bm\Psi$ has the form
\begin{equation}
\bm\Psi=\left\{
\begin{aligned}
& \bm\Psi^-=(U_1^-, U_2^-, U_3^-, P^-, S^-)(r, \theta),\quad\text{if}\,\, r_1< r<\xi(\theta), \,\, 0\leq \theta <\theta_0,\\
& \bm\Psi^+=(U_1^+, U_2^+, U_3^+, P^+, S^+)(r, \theta),\quad\text{if}\,\, \xi(\theta)< r<r_2, \,\, 0\leq \theta <\theta_0,
\end{aligned}
\right.
\end{equation}
and the following Rankine-Hugoniot conditions on the shock surface $\mcS=\{(r, \theta)| r=\xi(\theta)\}$ are satisfied
\begin{equation}\lab{rh-sph}
\begin{cases}
[\rho U_1] -\f{\xi'(\theta)}{\xi(\theta)} [\rho U_2]=0,\\
[\rho U_1^2+ P]-\f{\xi'(\theta)}{\xi(\theta)} [\rho U_1 U_2] =0,\\
[\rho U_1 U_2] -\f{\xi'(\theta)}{\xi(\theta)} [\rho U_2^2+P] =0,\\
[\rho U_1 U_3]-\f{\xi'(\theta)}{\xi(\theta)} [\rho U_2 U_3]=0,\\
[\mfe+\f{1}{2}|U|^2+\f{P}{\rho}]=0.
\end{cases}
\end{equation}

To state the main results, some weighted H\"{o}lder norms are needed. For any bounded domain $\mcD \subset \mbR^n$, $\mcK\subset \p\mcD$, and $\bx\in \mcD$, define
\be\no
\de_{\bx}:= \text{dist}({\bx},\mcK),\q \text{and}\q \de_{\bx,\ti{\bx}}:= \min(\de_{\bx},\de_{\ti{\bx}}).
\ee
For any nonnegative integer $m$, $\al\in (0,1)$ and $\sigma\in \mbR$, define weighted H\"{o}lder norms by
\be\no
[u]_{k,0; \mcD}^{(\si;\mcK)} &:=& \sum_{|\beta|=k} \sup_{\bx\in \mcD} \de_{\bx}^{\max\{|\beta|+\si, 0\}} |D^{\beta} u({\bx})|,\ k=0,1,\cdots, m,\\\no
[u]_{m,\al;\mcD}^{(\si;\mcK)} &:=& \sum_{|\beta|=m} \sup_{{\bx,\ti{\bx}}\in \mcD, {\bx}\neq \ti{{\bx}}}\de_{{\bx,\ti{\bx}}}^{\max\{m+\al+\si, 0\}} \f{|D^{\beta}u({\bx})-D^{\beta}u({\ti{\bx}})|}{|{\bx-\ti{\bx}}|^{\al}},\\\no
\|u\|_{m,\al;\mcD}^{(\si;\mcK)} &:=& \sum_{k=0}^m [u]_{k,0; \mcD}^{(\si;\mcK)} + [u]_{m,\al; \mcD}^{(\si;\mcK)}.
\ee
$C_{m,\al;\mcD}^{(\si;\mcK)}$ denotes the space of all smooth functions whose $\|\cdot\|_{m,\al;\mcD}^{(\si;\mcK)}$ norms are finite. One can refer to \cite{gh80,gt98,lieberman13} for the  properties of these weighted H\"{o}lder spaces.
Furthermore, $\Omega_\pm$ are defined as follows
\begin{equation*}
\Omega_-:=\{(r,\theta): r_1\leq r\leq \xi(\theta), 0\leq \theta< \theta_0+\epsilon f(r)\}\quad \text{and}\quad \Omega_+:= \Omega\setminus \Omega_-.
\end{equation*}

\bt\lab{transonic}
{\it Assume that $\Gamma$ satisfies  \eqref{wall} and $\bm{\Psi}_{en}$ satisfies \eqref{super3}. There exists a small $\e_0>0$ depending only on the background solution $\bPsi_b$ and boundary data $\bm{\Psi}_p$, $f$, $P_{0}$ such that if $0\leq \e<\e_0$, the problem \eqref{euler-sph1} with \eqref{super1},  \eqref{slip1}, \eqref{pressure}, and \eqref{rh-sph} has a unique solution $\bm{\Psi}^+=(U_1^+,U_2^+,U_3^+,P^+,S^+)(r,\theta)$ with the shock front $\mcS=\{(r, \theta): r=\xi(\theta), \theta\in [0, \theta_*]\}$ satisfying the following properties.
\begin{enumerate}
  \item[(i)] The function $\xi(\theta)\in C_{3,\al; (0,\theta_*)}^{(-1-\al;\{\theta_*\})}$ satisfies
  \be\lab{transonic01}
  \|\xi(\theta)-r_b\|_{3,\al; (0,\theta_*)}^{(-1-\al;\{\theta_*\})}\leq C_0\e,
  \ee
  where $(\xi(\theta_*),\theta_*)$ stands for the intersection circle of the shock surface with the nozzle wall and $C_0$ is a positive constant depending only on the supersonic incoming flow.
  \item[(ii)] The solution $\bm{\Psi}^+=(U_1^+,U_2^+,U_3^+,P^+,S^+)(r,\theta)\in C_{2,\al; \Omega_+}^{(-\al;\Ga_{w,s})}$ satisfies the entropy condition
  \begin{equation}\label{entropycond1}
  P^+(\xi(\theta)+,\theta) >P^-(\xi(\theta)-, \theta)\quad \text{for}\,\, \theta\in [0, \theta_*]
  \end{equation}
   and
  \be\lab{subsonic01}
  \|\bm{\Psi}^+ -\hat{\bm{\Psi}}^+_b\|_{2,\al; \Omega_+}^{(-\al;\Ga_{w,s})}\leq C_0\e,
  \ee
  where
  \be\no
  \Ga_{w,s}=\{(r,\theta): \xi(\theta)\leq r\leq r_2, \theta=\theta_0+\e f(r)\}.
  \ee
\end{enumerate}
}\et

In fact, if the nozzle boundary is straight and the exit pressure satisfies some further compatibility conditions, we have the higher order regularity for both the flows and the shock surface. This is our second main result.
\bt\lab{main2}
{\it Assume that the nozzle wall is straight, i.e., $f(r)\equiv 0$.
If, in addition to \eqref{super3}, the following compatibility conditions
\be\lab{pressure1}
P_{0}'(0)=P_{0}'(\theta_0)=0,
\ee
and
\be\lab{compati2}
\ U_{3,p}^-(\theta_0)=0,\q\f{d}{d\theta} (U_{1,p}^-, U_{3,p}^-,S_p^-)(\theta_0)=0,
\ee
hold
then the system \eqref{euler-sph1} in $\Omega_b$ together with \eqref{super1}, \eqref{pressure}, and the slip boundary conditions
\be\lab{wallslip}
U_2(r,\theta_0)= 0,\q r\in [r_1,r_2].
\ee
 has a unique solution
$\bm\Psi(r,\theta)$ with the shock surface $\mcS=\{(r, \theta): r=\xi(\th), \theta \in [0, \theta_0]\}$
satisfying the following properties.
\begin{enumerate}
  \item[(i)] The function $\xi(\th)\in C^{3,\al}([0,\theta_0])$ satisfies
  \be\lab{shock00}
  \|\xi(\th)-r_b\|_{C^{3,\al}([0,\theta_0])}\leq C_0\e,
  \ee
  where $C_0$ is a positive constant depending only on the supersonic incoming flow and the background solutions.
  \item[(ii)] ${\bm{\Psi}}^+= (U_1^{+},U_2^{+},U_3^{+}, P^{+}, S^{+})(r,\theta)\in C^{2,\al}(\ol{\msR_+})$ satisfies the entropy condition \eqref{entropycond1} with $\theta_*=\theta_0$ and
  \be\lab{subsoinc00}
  \|\bPsi(r,\theta)-\bPsi_b^+(r, \theta)\|_{C^{2,\al}(\ol{\msR_+})}\leq C_0\e,
  \ee
  where $\msR_+=\{(r,\th): \xi(\th)<r<r_2,0<\theta<\theta_0\}$ is the subsonic region.
\end{enumerate}
}\et

We make some comments on the key ingredients of the analysis in this paper. As is well-known, the supersonic flow is fully determined in the whole nozzle when the data at the entrance is given. Therefore, the transonic shock problem is reduced to a free boundary problem in subsonic region where the unknown shock surface is a free boundary and should be determined with the subsonic flow simultaneously, see \cite{lxy13}. In general, the optimal boundary regularity for subsonic flow is $C^{\al}$ for some $\al\in (0,1)$ (see \cite[Remark 3.2 and Lemma 3.3]{xyy09}), hence the streamline may not be uniquely determined. For two dimensional problem, the strategy to overcome this difficulty is to introduce a Lagrangian transformation to straighten the streamline. However, there is a singular term $\sin\th$ in the density equation (cf. \eqref{euler-sph1}) for axisymmetric flows. This makes the Lagrangian transformation (the one used in \cite{lxy13}) not invertible near the axis $\theta=0$. Our key observation is that the singular term $\sin\theta$ is of order $O(\theta)$ so that there is a simple invertible Lagrangian transformation to straighten the streamline.  Although the density equation still preserves the conservation form and a potential function as in \cite{lxy13} can be introduced, it is not easy to represent all the quantities in terms of the potential function and the entropy because the function $\theta$ becomes a nonlocal and nonlinear term in the Lagrangian coordinates. Here we resort to the first order elliptic system for the flow angle and the pressure and look for the solution in the function space $C_{2,\al;\Omega_+}^{(-\al;\Ga_{w,s})}$ rather than the space $C_{1,\al;\Omega_+}^{(-\al;\Ga_{w,s})}$ used in \cite{lxy13}. The axisymmetric Euler system with the shock front equation can be decomposed as a boundary value problem for a first order elliptic system with a nonlocal term and a singular term together with some transport equations. Compared with the elliptic system derived in \cite{lxy10b}, the coefficients for the linearized elliptic system for the angular velocity and pressure are smooth near the axis. One may refer to Proposition \ref{elliptic35} for more details. When the nozzle is a straight cone, even if  the swirl component of the velocity is not zero, the key issue is that $U_3=\p_{\th}U_3=0$ on the axis so that the singular term $\f{U_3^2\cot\theta}{r}$ does not cause any essential difficulty.

The rest of this paper is organized as follows. In Section 2, we introduce a new invertible Lagrangian transformation and reformulate the transonic shock problem in the new coordinates. Then the Euler system is decomposed as an elliptic system of the flow angle and the pressure together with the transport equations for the entropy, the swirl  velocity, and the Bernoulli function. An iteration scheme is developed in Section 3 to prove the existence and uniqueness of the transonic shock problem. In the last section, an improved regularity of the shock front and subsonic solutions is obtained if the nozzle is kept to be straight and some further compatibility conditions are satisfied.

\section{The reformulation of the transonic shock problem}
In this section, we first introduce a Lagrangian transformation to rewrite the Euler system. Then we use a transformation to fix the shock front so that the problem becomes a fixed boundary problem.
\subsection{Lagrangian formulation}
As we mentioned before, in general,
 one  can only expect the $C^{\al}$ boundary regularity for the solution in subsonic region (\cite[Remark 3.2]{xyy09}). To avoid the difficulty to determine the streamline uniquely, we  introduce a Lagrangian transformation to straighten the streamline. Note that there is a singular factor $\sin\theta$ in the density equation of \eqref{euler-sph1}, the standard Lagragian coordinates used in \cite{lxy13} is not invertible near the axis $\theta=0$. Observing that $\sin\theta$ is of order $O(\theta)$ near $\theta=0$, there indeed exists a simple invertible Lagrangian coordinates so that the streamlines can be straightened. Define $(\tilde{y}_1, \tilde{y}_2)=(r, \tilde{y}_2(r,\theta))$ such that
\be\lab{lagrange}
\left\{
\begin{aligned}
&\f{\p \tilde{y}_2}{\p r}= -r \rho^- U_2^- \sin \theta,\\
& \f{\p \tilde{y}_2}{\p\theta}= r^2 \rho^- U_1^-\sin\theta,\\
&\tilde{y}_2(r_1,0)=0,
\end{aligned}
\right.  \quad \text{and}\quad
\left\{
\begin{aligned}
&\f{\p \tilde{y}_2}{\p r}= -r \rho^+ U_2^+\sin\theta,\\
& \f{\p \tilde{y}_2}{\p\theta}= r^2 \rho^+ U_1^+\sin\theta,\\
&\tilde{y}_2(r_1,0)=0
\end{aligned}
\right.
\ee
for $(r,\theta)\in \ol{\Omega_-}$ and $\ol{\Omega_+}$, respectively.
It is clear that $\tilde{y}_2\geq 0$ in $\ol{\Omega}$ as long as $U_1^{\pm}>0$ in $\overline{\Omega^{\pm}}$.
On the axis $\theta=0$ and the nozzle wall $\Gamma$, one has
\be\no
\f{d}{d r} \tilde{y}_2(r,0)=0 \quad \text{and}\quad
\f{d}{dr} \tilde{y}_2(r,\theta_0 +\e f(r))=0.
\ee
Without loss of generality, assume that
\be\no
\tilde{y}_2(r,0)=0\q \text{for all } r\in [r_1,r_2].
\ee
Then there exist two positive constants $M$ and $M_1$ satisfying
\be\no
\tilde{y}_2(r,\theta_0+ \e f(r))=M^2\,\,\text{for }r \in [r_1,r_*] \text{ and } \tilde{y}_2(r,\theta_0+ \e f(r))=M_1^2 \,\,\text{for }r\in [r_*, r_2]
\ee
 respectively,
where $(r_*, \theta_0+ \e f(r_*))$ is the intersection point of the shock front $\mcS$ with the nozzle wall $\Gamma$. We claim that $\tilde{y}_2(r,\theta)$ is well-defined in $\bar{\Omega}$ and belongs to $\text{Lip}(\bar{\Omega})$. Using the first equation in \eqref{rh-sph} yields
\be\no
\f{d}{d\theta}\tilde{y}_2(\xi(\th)+0,\th)= \f{d}{d\theta}\tilde{y}_2(\xi(\th)-0,\th).
\ee
This implies  $M_1=M$ which can be computed as follows
\be\no
M^2= r_1^2 \int_0^{\theta_0} (\rho^- U_1^-)(r_1,\theta)\sin\theta d\theta>0.
\ee

 Set
\begin{eqnarray}\label{Lagtransform}
y_1=r,\ \ y_2= \tilde{y}_2^{\frac{1}{2}}(r,\theta).
\end{eqnarray}
Under the transformation \eqref{Lagtransform}, the domains $\Omega$, $\Omega_-$, and $\Omega_+$ are changed into $D=(r_1,r_2)\times (0,M)$,
\be\lab{domain}
D_-=\{(y_1,y_2): r_1<y_1<\psi(y_2), y_2\in (0,M)\}, \quad \text{and}\quad D_+=D\setminus \overline{D_-},
\ee
respectively. Note that if $(\rho^{\pm}, U_1^{\pm}, U_2^{\pm})$ are close to the background solution $(\rho_b^{\pm}, U_b^{\pm},0)$, then there exist two positive constants $C_1$ and $C_2$  depending only on the background solution such that
\be\no
C_1 \theta^2 \leq \tilde{y}_2(r,\theta)= r^2 \int_0^{\theta} (\rho^{\pm} U_1^{\pm})(r,\tau) \sin \tau d\tau \leq C_2 \theta^2.
\ee
Hence $\sqrt{C_1} \theta\leq y_2(r,\theta)\leq \sqrt{C_2} \theta$ and the Jacobian of the transformation $\mc{L}: (r,\theta)\in \bar{\Omega}\mapsto (y_1,y_2)=(r,y_2(r,\theta))\in \bar{D}$ satisfies
\be\lab{jacobian}
\det \left(\begin{array}{ll}
\f{\p y_1}{\p r} & \f{\p y_1}{\p \theta}\\
\f{\p y_2}{\p r} & \f{\p y_2}{\p \theta}
\end{array}\right)= \det \left(\begin{array}{ll}
\q 1\q & \q 0\q\\
-\f{r \rho U_2\sin\theta}{ 2 y_2} & \f{r^2 \rho U_1 \sin\theta}{2 y_2}
\end{array}\right)= \f{r^2 \rho U_1 \sin\theta}{2 y_2}\geq C_3>0,
\ee
where $C_3$ is a constant depending only on the background solution. Hence the inverse transformation $\mc{L}^{-1}:(y_1,y_2)\mapsto (r,\theta)$ exists.
To simplify the notations, we neglect the superscript ``+" for the solutions in the subsonic region.
Under the transformation \eqref{lagrange}, the Euler system  \eqref{euler-sph1} can be written as
\be\lab{euler-lag'}\bca
\p_{y_1}\b(\f{2 y_2}{y_1^2 \rho U_1\sin\theta}\b)- \p_{y_2}\b(\f{U_2}{y_1 U_1}\b)=0,\\
\p_{y_1}(U_1+\f{P}{\rho U_1})- \f{y_1\sin\theta}{2 y_2} \p_{y_2} (\f{P U_2}{U_1})-\f{2 P}{y_1\rho U_1}-\f{P U_2\cos\theta}{y_1\rho U_1^2\sin\theta}-\f{ (U_2^2+U_3^2)}{y_1 U_1}=0,\\
\p_{y_1}(y_1 U_2) + \f{y_1^2\sin\theta}{2 y_2}\p_{y_2} P -\f{U_3^2}{U_1}\cot\theta=0,\\
\p_{y_1}(y_1 U_3\sin\theta)=0,\\
\p_{y_1} B=0.
\eca\ee
The nozzle wall $\Ga_{w,s}$ is straightened to be $\Ga_{w,y}=(\psi(M),r_2)\times \{M\}$. Suppose that the shock front $\mc{S}$ and the flows ahead and behind $\mc{S}$ are denoted by $y_1=\psi(y_2)$ and $(U_1^{\pm},U_2^{\pm},U_3^{\pm},P^{\pm},S^{\pm})(y)$, respectively. Then the Rankine-Hugoniot conditions on $\mc{S}$, \eqref{rh-sph}, become
\be\lab{rh10}\bca
\f{2y_2}{\psi(y_2)\sin\theta}[\f{1}{\rho U_1}] + \psi'(y_2) [\f{U_2}{U_1}]=0,\\
[U_1 + \f{P}{\rho U_1}] + \psi'(y_2)\f{\psi(y_2)\sin\theta}{2y_2} [ \f{P U_2}{U_1}]=0,\\
[U_2]- \psi'(y_2)\f{\psi(y_2)\sin\theta}{2y_2}[P]=0,\\
[U_3]=0,\\
[B]=0,
\eca\ee
where $[g]=g(\psi(y_2)+, y_2)-g(\psi(y_2)-, y_2)$.

It should be emphasized that in terms of the new coordinates $(y_1, y_2)$,  $\theta$  becomes nonlinear and nonlocal. Indeed, one has
\be\lab{theta}
\f{\p\theta}{\p y_1}=\f{U_2}{y_1 U_1},\q \f{\p\theta}{\p y_2}=\f{2 y_2}{y_1^2\rho U_1\sin\theta},\q \theta(y_1,0)= 0.
\ee
Thus  it holds that
\be\lab{theta1}
\theta(y_1,y_2)= \arccos \b(1-\int_0^{y_2}\f{2s}{y_1^2 (\rho U_1)(y_1,s)} ds\b).
\ee
For the background solution $(\rho_b^{\pm}, U_b^{\pm})$, the similar Lagrangian transformation yields
\be\no
\f{\p\theta_b}{\p y_2}= \f{2 y_2}{y_1^2(\rho_b U_b)(y_1)\sin\theta}= \f{2\ka_b y_2}{\sin\theta_b},
\ee
where
\begin{equation}\label{defkb}
\ka_b=\f{1}{y_1^2 (\rho_b U_b)(y_1)}
\end{equation}
 is a positive constant for any $y_1\in [r_b,r_2]$. Hence
\be\lab{background-theta}
\theta_b(y_2)= \arccos (1-\ka_b y_2^2).
\ee
\subsection{The elliptic modes}
 Note that there is a singular factor $\cot\theta$ in \eqref{euler-lag'}, which is also a nonlinear and nonlocal term because of \eqref{theta1}. In order to study the system \eqref{euler-lag'}, we need to focus on the governing equations for  the pressure and the flow angle. Denote $\varpi=\f{U_2}{U_1}$. Due to the first equation in  \eqref{euler-lag'},  the second and third equations in \eqref{euler-lag'} can be written as
\be\lab{elliptic11}
\left\{
\begin{aligned}
&\p_{y_1} \varpi - \f{y_1 \rho U_1\varpi\sin\theta}{2 y_2} \p_{y_2} \varpi- \f{\varpi}{y_1} -\f{\varpi^2}{y_1}\cot\theta + \f{y_1\sin\theta}{2 y_2 U_1} \p_{y_2} P\\
&\qquad - \f{\varpi}{\rho c^2(\rho, S)} \p_{y_1} P- \f{U_3^2}{y_1 U_1^2} \cot\theta=0,\\
&\p_{y_1} P- \f{\rho c^2(\rho,S) U_1^2}{y_1(c^2(\rho,S)-U_1^2)} \f{y_1^2 \rho U_1\sin\theta}{2y_2}\p_{y_2}\varpi-\f{y_1\rho c^2(\rho,S)U_1 \varpi\sin\theta}{2 y_2(c^2(\rho,S)-U_1^2)} \p_{y_2}P \\
& \q\q-\f{\rho c^2(\rho,S)U_1^2}{y_1(c^2(\rho,S)-U_1^2)}(\varpi^2+ \varpi \cot\theta+2)-\f{\rho c^2(\rho,S)U_3^2}{y_1(c^2(\rho,S)-U_1^2)}=0,
\end{aligned}
\right.
\ee
where one used the following equation for the entropy,
\be\lab{entropy1}
\p_{y_1} S=0.
\ee
In fact, the equation \eqref{entropy1} can be obtained from \eqref{euler-lag'} together with the definition of the equation of the state \eqref{eqstate}.
It follows from \eqref{slip1} and \eqref{pressure}  that the corresponding boundary conditions for $\varpi$ and $P$ read
\be\lab{boundary11}\begin{cases}
\varpi(y_1,0)=0,\q \varpi(y_1,M)= \e y_1 f'(y_1),\q \text{for any } y_1\in [r_1, r_2],\\
P(r_2,y_2)= P_e + \e P_0(\theta(r_2,y_2)),\q \text{for any } y_2\in [0,M].
\end{cases}\ee

By the third equation in \eqref{rh10}, one has
\be\lab{shock11}
\psi'(y_2)=\f{2y_2}{\sin\theta(\psi(y_2),y_2)}\f{U_2(\psi(y_2),y_2)-U_2^-(\psi(y_2),y_2)}{\psi(y_2)(P(\psi(y_2),y_2)-P^-(\psi(y_2),y_2))}.
\ee
Substituting \eqref{shock11} into the first two equations in \eqref{rh10} yields that
\be\lab{rh12}\begin{cases}
[\rho U_1]= \rho U_1 \rho^- U_1^- \f{[U_2]}{[P]} \left[\f{U_2}{U_1}\right],\\
[\rho U_1^2+ P]= -\rho^- U_1^-  \f{[U_2]}{[P]} \left[\f{P U_2}{U_1}\right]+ (\rho (U_1)^2+P) \rho^- U_1^-  \f{[U_2]}{[P]} \left[\f{U_2}{U_1}\right].
\end{cases}
\ee
Furthermore, the last two equations in \eqref{rh10} are equivalent to
\be\lab{rh13}
\begin{aligned}
&U_3(\psi(y_2),y_2)= U_3^-(\psi(y_2),y_2)\quad\text{and}\quad  B(\psi(y_2),y_2)= B^-(\psi(y_2),y_2).
\end{aligned}
\ee

It follows from the Bernoulli's law, the last equation in \eqref{euler-lag'}, that one can represent $U_1$ as
\be\no
U_1=\sqrt{\f{2 B-U_3^2-\f{2 A^{\f1{\ga}}\ga}{\ga-1} P^{\f{\ga-1}{\ga}}e^{\frac{S}{\ga c_{v}}}}{1+\varpi^2}}.
\ee
Hence we can write $\rho U_1$ and $\rho U_1^2+P$ as smooth functions of $P$, $S$, $B$, $U_3$, and $\varpi$.
Note that
\[
(\rho_b^+U_b^+)(r_b)=(\rho_b^-U_b^-)(r_b)\quad\text{and}\quad  (\rho_b^+(U_b^+)^2+P_b^+)(r_b)=(\rho_b^-(U_b^-)^2+P_b^-)(r_b)
\]
Applying the Taylor's expansion for \eqref{rh12} yields
\be\label{BCPS}
\left\{
\begin{aligned}
&a_{11}(P(\psi(y_2),y_2)-P_b^+(r_b)) + a_{12} (S(\psi(y_2),y_2)-S_b^+) \\
&\q= - \f{\rho_b^+(r_b)}{U_b^+(r_b)}(B(\psi(y_2),y_2)-B_b^+)-\f{2 (\rho_b^- U_b^-)(r_b)}{r_b} (\psi(y_2)-r_b)+ R_1,\\
&a_{21}(P(\psi(y_2),y_2)-P_b^+(r_b)) + a_{22} (S(\psi(y_2),y_2)-S_b^+) \\
&\q= - 2\rho_b^+(r_b)(B(\psi(y_2),y_2)-B_b^+)-\f{2 (\rho_b^- (U_b^-)^2)(r_b)}{r_b} (\psi(y_2)-r_b)+ R_2,
\end{aligned}
\right.
\ee
where
\be\no
&&a_{11}=\f{(U_b^+(r_b))^2-c^2(\rho_b^+(r_b),S_b^+)}{U_b^+(r_b)c^2(\rho_b^+(r_b),S_b^+)},\ a_{12}=-\f{(U_b^+(r_b))^2+ \f{1}{\ga-1}c^2(\rho_b^+(r_b),S_b^+)}{c_v U_b^+(r_b)c^2(\rho_b^+(r_b),S_b^+)}P_b^+(r_b),\\\no
&&a_{21}=\f{(U_b^+(r_b))^2-c^2(\rho_b^+(r_b),S_b^+)}{c^2(\rho_b^+(r_b),S_b^+)},\ a_{22}=-\f{(U_b^+(r_b))^2+ \f{2}{\ga-1}c^2(\rho_b^+(r_b),S_b^+)}{c_v c^2(\rho_b^+(r_b),S_b^+)}P_b^+(r_b)
\ee
and $R_i=R_i (\bm{\Phi}^+(\psi(y_2),y_2)-\bm{\Phi}_b^+(r_b), \psi(y_2)-r_b,\bm{\Phi}^-(\psi(y_2),y_2)-\bm{\Phi}_b^-(\psi(y_2)))$  ($i=1$, $2$) denotes the error term with
\begin{equation}
\begin{aligned}
\bm{\Phi}^{\pm}:=& (U_1^{\pm},\varpi^{\pm},U_3^{\pm},P^{\pm},S^{\pm})\quad \text{and}\quad  \bm{\Phi}_b^{\pm}:= (U_b^{\pm},0,0,P_b^{\pm},S_b^{\pm})
\end{aligned}
\end{equation}
Later on, we denote $\bPhi^+$ by $\bPhi$ for simplicity.
Furthermore, for $i=1$ and $2$, straightforward computations give
\begin{equation}\label{estRi}
|R_i|\leq C (|\bm{\Phi}(\psi(y_2),y_2)-\bm{\Phi}_b^+(r_b)|^2+|\psi(y_2)-r_b|^2+|\bm{\Phi}^-(\psi(y_2),y_2)-\bm{\Phi}_b^-(\psi(y_2))|).
\end{equation}
 It follows from \eqref{com-euler} and \eqref{RH_background} that $B_b^+=B_b^-$. This, together with \eqref{rh13}, yields
\[
B(\psi(y_2),y_2)-B_b^+= B^-(\psi(y_2),y_2)-B_b^-.
\]
Hence one has
\be\lab{rh15}\begin{cases}
P(\psi(y_2),y_2)-P_b^+(r_b)= e_1 (\psi(y_2)-r_b)+ R_3,\\
S(\psi(y_2),y_2)-S_b^+= e_2 (\psi(y_2)-r_b) +R_4,
\end{cases}\ee
where $R_i$ ($i=3$, $4$) satisfies the similar estimate as \eqref{estRi},
\be
\begin{aligned}
e_1=&2\f{c_v (\rho_b^- U_b^-)(r_b) c^2(\rho_b^+(r_b),S_b^+)}{r_b((U_b^+(r_b))^2-c^2(\rho_b^+(r_b),S_b^+))}\b(U_b^-(r_b)\b((U_b^+(r_b))^2+\f{1}{\ga-1}c^2(\rho_b^+(r_b),S_b^+)\b)\\\no
&\q\q-U_b^+(r_b)\b((U_b^+(r_b))^2+\f{2}{\ga-1}c^2(\rho_b^+(r_b),S_b^+)\b)\b),
\end{aligned}
\ee
and
\be\label{defe2}
e_2= \f{2(\ga-1)c_v}{r_b} \f{(\rho_b^- U_b^-)(r_b)}{P_b^+(r_b)}(U_b^-(r_b)-U_b^+(r_b)).
\ee
Clearly, $e_2>0$.
\subsection{Fix the domain and the reformulation of the problem}
To fix the shock front, we introduce the following  coordinate transformation
\be\no
z_1=\f{y_1-\psi(y_2)}{r_2-\psi(y_2)} N\quad \text{and}\quad  z_2=y_2 \quad \text{with}\quad  N=r_2-r_b.
\ee
Clearly, the domain $D_+$ and the wall $\Ga_{w,y}$ are changed into
\be\no
E_+=(0,N)\times (0,M)\quad  \text{and} \quad \Ga_{w,z}=(0,N)\times \{M\},
\ee
respectively.
Define
\be\no
&&(\ti{\rho}_b^+,\ti{U}^+_b,\ti{P}_b^+)(z_1)= (\rho_b^+, U_b^+, P_b^+)(r_b+z_1),\\\no
&& (\tilde{\rho},\tilde{U}_1, \ti{\varpi}, \tilde{U}_3, \ti{P},\ti{S},\ti{B},\ti{\theta})(z)= (\rho,U_1, \varpi, U_3, P,S, B,\theta)\b(\psi(z_2)+\f{r_2-\psi(z_2)}{N}z_1,z_2\b).
\ee
Set $\bW:=(W_1, W_2, W_3, W_4, W_5, W_6)$ with
\be\no
&&W_1(z)= \ti{U}_1(z)-\ti{U}^+_b(z_1),\ W_2(z)=\ti{\varpi}(z),\quad W_3(z)=\tilde{U}_3(z),\\\no
&&W_4(z)=\ti{P}(z)-\ti{P}_b^+(z_1),\ W_5(z)= \ti{S}(z)- S_b^+,\  W_6(z_2)=\psi(z_2)-r_b,
\ee
and
\begin{equation}
W_6^\diamondsuit(z_2)=r_b+W_6(z_2),\quad
W_6^\#(z_1, z_2) = r_b+z_1+\f{N-z_1}{N}W_6(z_2).
\end{equation}
In terms of the coordinates $(z_1, z_2)$, the equation \eqref{shock11} becomes
\be\lab{shock12}
W_6'(z_2)=\f{2z_2}{\sin\theta(0,z_2)}\f{(\ti{U}_b^+(0)+W_1(0,z_2))W_2(0,z_2)-U_2^-(W_6^\diamondsuit(z_2),z_2)}
{W_6^\diamondsuit(z_2)((\ti{P}_b^+(0)+W_4(0,z_2))-P^-(W_6^\diamondsuit(z_2),z_2))}.
\ee

It follows from the last equation in \eqref{euler-lag'} and \eqref{entropy1} that one has
\be\lab{entropy11}
\p_{z_1} W_5=0\q\text{and}\q \p_{z_1} \tilde{B}=0,\ \text{ in } E_+.
\ee
This, together with \eqref{rh13} and the second equation in \eqref{rh15}, gives
\begin{equation}\lab{entropy12}
\begin{aligned}
W_5(z)=& W_5(0,z_2)
=  e_2 W_6(z_2) \\
&+R_4(\bm{\Phi}(W_6^\diamondsuit(z_2),z_2)-\bm{\Phi}_b^+(r_b), W_6(z_2),\bm{\Phi}^-(W_6^\diamondsuit(z_2),z_2)-\bm{\Phi}_b^-(W_6^\diamondsuit(z_2))),\\
\end{aligned}
\end{equation}
and
\be\lab{bernoulli10}
B(z)-B_b^+= B(0,z_2)-B_b^+= B^-(W_6^\diamondsuit(z_2),z_2)-B_b^-.
\ee
It follows from the fourth equations in \eqref{euler-lag'} and \eqref{rh10} that
\be\lab{swirl11}\begin{cases}
\p_{z_1}[W_6^\#(z_1, z_2)W_3\sin\theta(z_1,z_2)]=0,\\
W_3(0,z_2)= U_3^-(W_6^\diamondsuit(z_2),z_2).
\end{cases}\ee
This yields
\be\lab{swirl12}
W_3(z)&=& \f{W_6^\diamondsuit(z_2)}{W_6^\#(z_1,z_2)} \f{\sin\theta(0,z_2)}{\sin\theta(z_1,z_2)} U_3^-(W_6^\diamondsuit(z_2),z_2).
\ee

 Note that
\be\no
U_1(y_1,y_2)= (\ti{U}_b^+ + W_1)\left(\f{y_1-W_6^\diamondsuit(y_2)}{N-W_6(y_2)}N, y_2\right).
\ee
Then it follows from \eqref{theta1} that
\be\lab{theta11}
\begin{aligned}
\theta(z_1,z_2)=& \arccos(1-\vartheta(z_1,z_2)),
\end{aligned}
\ee
where
\begin{equation}
\vartheta(z_1,z_2)= \int_0^{z_2} \f{2s}{(W_6^\#(z_1,z_2))^2\{\varrho(W_4, W_5)(\ti{U}_b^+ +W_1)\}\b(\f{W_6^\#(z_1,z_2)-W_6^\diamondsuit(s)}{N-W_6(s)}N,s\b)}ds
\end{equation}
with
\begin{equation}\label{defvarrho}
\varrho(W_4,W_5)= A^{-\f1{\ga}} (\tilde{P}_b^++ W_4)^{\f1{\ga}} e^{-\f{S_b^++W_5}{\ga c_v}}.
\end{equation}

The Bernoulli's law \eqref{bernoulli10} together with the Rankine-Hugoniot conditions \eqref{rh13} yields
\begin{equation}\lab{bernoulli11}
\begin{aligned}
&\left\{\f12 (\ti{U}_b^+ + W_1)^2 (1+W_2^2) +\f12 W_3^2 + h(\ti{P}_b^+ + W_4, S_b^+ +W_5)\right\}(W_6^\diamondsuit(z_2), z_2)\\
=& B^-(W_6^\diamondsuit(z_2),z_2).
\end{aligned}
\end{equation}
Since $B_b^-=B_b^+= \f12(\ti{U}_b^+)^2+h(\ti{P}_b^+, S_b^+)$, one has
\begin{equation}\lab{bernoulli12}
\begin{aligned}
W_1=&\f{1}{\ti{U}_b^+}\b\{B^-(W_6^\diamondsuit(z_2),z_2)-B_b^- -[h(\ti{P}_b^+ + W_4, S_b^+ +W_5)-h(\ti{P}_b^+, S_b^+)]\b\}\\
&\q \q-\f{1}{2 \ti{U}_b^+}[W_1^2 + (\ti{U}_b^++ W_1)^2 W_2^2+W_3^2].
\end{aligned}
\end{equation}

Finally, we rewrite the system \eqref{elliptic11} in terms of $W_2$ and $W_4$.
Note that
\be\lab{background}
\f{d}{dz_1}\ti{P}_b^+ -\f{2\ga \ti{P}_b^+(\ti{U}_b^+)^2}{(r_b+z_1)(c^2(\ti{\rho}_b^+,S_b^+)-(\ti{U}_b^+)^2)}=0.
\ee
Then straightforward calculations yield that
\be\no
&\q&-\f{2\ga \tilde{P} \tilde{U}_1^2}{\b(\psi(z_2)+\f{r_2-\psi(z_2)}{N} z_1\b)(c^2(\tilde{\rho},\tilde{S})-\tilde{U}_1^2)}+\f{2\ga}{r_b+z_1} \f{\ti{P}_b^+(\ti{U}_b^+)^2}{c^2(\ti{\rho}_b^+,S_b^+)-(\ti{U}_b^+)^2}\\\no
&=&e_3(z_1) (\tilde{B}(z)-B_b^+)+ e_4(z_1) W_4 + e_5(z_1) W_5 +\ti{e}_6(z_1) W_6(z_2)+ R_5({\bf W}),
\ee
where
\be\no
e_3(z_1)&=& \f{4\ga \ti{P}_b^+ c^2(\ti{\rho}_b^+, S_b^+)}{c^2(\ti{\rho}_b^+, S_b^+)-(\ti{U}_b^+)^2},\\\no
e_4(z_1)&=&\f{2\ga}{(r_b+z_1)\ti{\rho}_b^+(c^2(\ti{\rho}_b^+, S_b^+)-(\ti{U}_b^+)^2)}(\ti{\rho}_b^+(\ti{U}_b^+)^4- P_b^+ (\ti{U}_b^+)^2+ 2 \ti{P}_b^+c^2(\ti{\rho}_b^+, S_b^+)),\\\no
e_5(z_1)&=& \f{2\ga (\ti{P}_b^+)^2 ((\ti{U}_b^+)^2+\f{2}{\ga-1}c^2(\ti{\rho}_b^+, S_b^+))}{c_v(r_b+z_1)\ti{\rho}_b^+(c^2(\ti{\rho}_b^+, S_b^+)-(\ti{U}_b^+)^2)^2},\\\no
\ti{e}_6(z_1)&=&\f{2\ga(N-z_1) \ti{P}_b^+ (\ti{U}_b^+)^2}{N(r_b+z_1)^2(c^2(\ti{\rho}_b^+, S_b^+)-(\ti{U}_b^+)^2)},
\ee
and $R_5$ is quadratic with respect to ${\bf W}$. Clearly, one has
\[
e_3,\, e_4,\, e_5>0.
\]
Therefore, it follows from \eqref{elliptic11} that
\begin{equation}\lab{elliptic13}
\left\{
\begin{aligned}
&\p_{z_1} W_2-\f{c^2(\ti{\rho}_b^+,S_b^+)+(\ti{U}_b^+)^2}{(r_b+z_1)(c^2(\ti{\rho}_b^+,S_b^+)-(\ti{U}_b^+)^2)}W_2 + \f{r_b+z_1}{\ti{U}_b^+}\f{\sin\theta_b(z_2)}{2z_2} \p_{z_2} W_4 \\
&\q\q+ \f{r_b+z_1}{\ti{U}_b^+}\f{z_1-N}{N}\f{d}{d z_1}\ti{P}_b^+ \f{\sin\theta_b(z_2)}{2z_2} W_6'(z_2)= F_1({\bf W}, \nabla {\bf W}, \bm{\Phi}^- -\bm{\Phi}_b^-),\\
&\p_{z_1}W_4-\f{\ga\ti{P}_b^+(\ti{U}_b^+)^2}{c^2(\ti{\rho}_b^+,S_b^+)-(\ti{U}_b^+)^2}\f{1}{\ka_b(r_b+z_1)}\f{\sin\theta_b(z_2)}{2z_2} \b(\p_{z_2}W_2+ \f{2\ka_b z_2\cos\theta_b(z_2)}{\sin^2\theta_b(z_2)}W_2\b)\\
&\q\q+e_4(z_1) W_4(z)+ e_5(z_1) W_5(z)+e_6(z_1)W_6(z_2)=F_2({\bf W}, \nabla {\bf W}, \bm{\Phi}^- -\bm{\Phi}_b^-)
\end{aligned}
\right.
\end{equation}
where
$F_1({\bf W}, \nabla {\bf W}, \bm{\Phi}^- -\bm{\Phi}_b^-)$
and
$F_2({\bf W}, \nabla {\bf W}, \bm{\Phi}^- -\bm{\Phi}_b^-)$ are quadratic with respect to $\bW$ and $\nabla \bW$
and
\be\no
e_6(z_1)&=&\ti{e}_6(z_1)+\f{1}{N}\f{d}{dz_1}\ti{P}_b^+(z_1)=\f{2\ga r_2 \ti{P}_b^+ (\ti{U}_b^+)^2}{N(r_b+z_1)^2(c^2(\ti{\rho}_b^+, S_b^+)-(\ti{U}_b^+)^2)}.
\ee
Clearly, the system \eqref{elliptic13} should be supplemented with the following boundary conditions
\begin{equation}\label{BCW24}
\begin{aligned}
&W_4(0,z_2)= e_1 W_6(z_2) +R_3({\bf W}(0,z_2),\bm{\Phi}^--\bm{\Phi}_b^-),\\
&W_2(z_1,0)=0,\q \text{for } z_1\in [0,N],\\
&W_2(z_1,M)= \e  W_6^\#(r_1,M)) f'(W_6^\#(r_1,M)),\q \text{for } z_1\in [0,N],\\
&W_4(N,z_2)=\e P_0(\theta(N,z_2)),\q \text{for}\,\, z_2\in [0,M].
\end{aligned}
\end{equation}
Therefore, the original problem is equivalent to \eqref{shock12}, \eqref{entropy12},  \eqref{swirl12}, \eqref{bernoulli12}, and \eqref{elliptic13}-\eqref{BCW24}.

\section{Iteration scheme and Proof of Theorem \ref{transonic}}\noindent

We are now in position to design an iteration scheme to prove Theorem \ref{transonic}. The approach is motivated by \cite{lxy13}.  Define
\be\lab{class}
\Xi_{\de}=\left\{ {\bf W}\left| \begin{aligned}
&\||{\bf W}|\| \leq \delta; \,\, \p_{z_2} W_j(z_1,0)=0, j=1,3,4,5;\\
  & W_2(z_1,0)=\p_{z_2}^2 W_2(z_1,0)=W_5(z_1,0)=0;\,\, W_6'(0)=W_6^{(3)}(0)=0
\end{aligned}
\right.
\right\},
\ee
where
\[
\||{\bf W}|\|= \sum_{i=1}^5 \|W_i\|_{2,\al; E_+}^{(-\al;\Ga_{w,z})}+ \|W_6\|_{3,\al; (0,M)}^{(-1-\al;\{M\})}.
\]
Clearly, $\Xi_\de$ is a complete metric space under the metric $d(\bW, \hat{\bW})=\||\bW-\hat{\bW}|\|$.
 Given any $\hat{{\bf W}}\in \Xi_{\de}$, we  use an iteration to define a mapping with $\mc{T} \hat{\bf{W}}={\bf W}$   from $\Xi_{\de}$ to itself by choosing suitable small $\de$.

\subsection{The iteration scheme for $W_6$, $W_5$, and $W_3$}
It follows from \eqref{shock12}  that $W_6$ is required to satisfy the following equation
\be\no
&&W_6'(z_2)= a \f{2z_2}{\sin\theta_b(z_2)} W_2(0,z_2)+ R_{11}(\hat{{\bf W}}(0,z_2),\bm{\Phi}^-(\hat{W}_6^\diamondsuit(z_2),z_2)-\bm{\Phi}^-_b(\hat{W}_6^\diamondsuit(z_2))),
\ee
where $\hat{W}_6^\diamondsuit(z_2)=\hat{W}_6(z_2)+r_b$,  $R_{11}$ is quadratic with respect to $\hat{\bW}(0, z_2)$, and
\be\label{defa}
&&a=\f{\ti{U}_b^+(0)}{r_b(\ti{P}_b^+(0)-P^-_b(r_b))}.
\ee
Hence $W_6$ can be solved as follows
\be\lab{shock23}
W_6(z_2)= W_6(M)-a\int_{z_2}^M \f{2s}{\sin\theta_b(s)} W_2(0,s)ds +R_{12},
\ee
where $\theta_b$ is defined in \eqref{background-theta} and
\be\no
R_{12}(\hat{{\bf W}},\bm{\Phi}^--\bm{\Phi}^-_b)=-\int_{z_2}^M R_{11}(\hat{{\bf W}}(0,s), \bm{\Phi}^-(\hat{W}_6^\diamondsuit(s),s)-\bm{\Phi}^-_b(r_b)) ds.
\ee
We also note that for $\hat{{\bf W}}\in \Xi_{\de}$, $R_{11}(z_1,0)=\p_{z_2}^2 R_{11}(z_1,0)=0$ for any $z_1\in [0,N]$.

 Since $\p_{z_1} W_5=0$, one has
\be\lab{entropy21}
W_5(z)= W_5(0,z_2)
= e_2W_6(z_2)+ R_{4}(\hat{\bW}, \bPhi^--\bPhi_b^-),
\ee
where $e_2$ is defined in \eqref{defe2}.
It is easy to verify that $\p_{z_2} R_4(z_1,0)=0$ for $\hat{{\bf W}}\in \Xi_{\de}$.

It follows from \eqref{swirl12} that one defines
\be\lab{swirl22}
W_3(z_1,z_2)&=&\f{\hat{W}_6^\diamondsuit(z_2)}{\hat{W}_6^\#(z_1,z_2)} \f{\sin\hat{\theta}(0,z_2)}{\sin\hat{\theta}(z_1,z_2)} U_3^-(\hat{W}_6^\diamondsuit(z_2),z_2),
\ee
where
$\hat{W}_6^\#(z_1, z_2) = r_b+z_1+\f{N-z_1}{N}\hat{W}_6(z_2)$ and $\hat{\theta}(z_1, z_2) =\arccos (1-\hat{\vartheta}(z_1, z_2))$ with
\begin{equation}\lab{theta21}
\hat{\vartheta}(z_1,z_2)= \int_0^{z_2} \f{2s}{(\hat{W}_6^\#(z_1, z_2))^2\b\{\varrho(\hat{W}_4,\hat{W}_5)(\ti{U}_b^+ +\hat{W}_1)\b\}\b(\f{\hat{W}_6^\#(z_1,z_2)-\hat{W}_6^\diamondsuit(s)}{N-\hat{W}_6(s)}N,s\b)}ds,
\end{equation}
where $\varrho$ is the function defined in \eqref{defvarrho}.
 Note that $\f{\hat{W}_6^\#(z_1,z_2)-\hat{W}_6^\diamondsuit(s)}{N-\hat{W}_6(s)}N$ may exceed the interval $[0,N]$, hence we  extend the functions $\hat{{\bf W}}$ to a larger domain $[-N, 2N]\times [0,M]$ as follows
\be\lab{extension}
\hat{{\bf W}}^e(z_1,z_2)=\begin{cases}
\sum_{k=1}^3 c_k\hat{{\bf W}}(-\f{z_1}{k},z_2),&\q -N\leq z_1<0,\\
\sum_{k=1}^3 c_k\hat{{\bf W}}(\f{2N-z_1}{k},z_2),&\q N<z_1\leq 2N,
\end{cases}\ee
where the constants $c_k$ ($k=1,2,3$) satisfy the following algebraic relations
\be\lab{vandermon}
\sum_{k=1}^3 c_k= 1,\ \ \ -\sum_{k=1}^3 \frac{c_k}{k} =1,\ \ \ \sum_{k=1}^3 \f{c_k }{k^2}=1.
\ee
 It is easy to see that the extended functions  $\hat{{\bf W}}^e$ belong to $C^2$ as long as $\hat{{\bf W}}\in C^2$. For ease of notations, we still denote these extended functions by $\hat{{\bf W}}$.

\subsection{The iteration scheme for $W_2$ and $W_4$}
Substituting \eqref{shock23} and \eqref{entropy21} into \eqref{elliptic13} yields that $W_2$ and $W_4$ satisfy the following first order elliptic system with a nonlocal term and a parameter,
\begin{equation}\lab{elliptic23}\begin{cases}
\p_{z_1} W_2-\f{c^2(\ti{\rho}_b^+,S_b^+)+(\ti{U}_b^+)^2}{(r_b+z_1)(c^2(\ti{\rho}_b^+,S_b^+)-(\ti{U}_b^+)^2)}W_2 + \f{r_b+z_1}{\ti{U}_b^+}\f{\sin\theta_b(z_2)}{2z_2} \p_{z_2} W_4 + a\f{r_b+z_1}{\ti{U}_b^+}\f{z_1-N}{N}\f{d}{d z_1}\ti{P}_b^+ W_2(0,z_2)\\
\q\q\q= F_3(\hat{{\bf W}},\nabla\hat{{\bf W}},\bm{\Phi}^--\bm{\Phi}^-_b),\\
\p_{z_1}W_4-\f{\ga\ti{P}_b^+(\ti{U}_b^+)^2}{\ka_b(r_b+z_1)(c^2(\ti{\rho}_b^+,S_b^+)-(\ti{U}_b^+)^2)}\f{\sin\theta_b(z_2)}{2z_2}\b(\p_{z_2}W_2
+\f{2\ka_b z_2\cos\theta_b(z_2)}{\sin^2\theta_b(z_2)} W_2\b)+r_4(z_1)W_4\\
\q+\b(e_6(z_1)+ e_2 e_5(z_1)\b) \b(W_6(M)-a\int_{z_2}^M \f{2s} {\sin\theta_b(s)}W_2(0,s)ds\b)= F_4(\hat{{\bf W}},\nabla\hat{{\bf W}},\bm{\Phi}^--\bm{\Phi}^-_b),\\
W_4(0,z_2)=  e_1 \b(W_6(M)-a\int_{z_2}^M \f{2s}{\sin\theta_b(s)} W_2(0,s)ds\b)+ e_1 R_{12} +R_5(\hat{{\bf W}}(0,z_2),\bm{\Phi}^--\bm{\Phi}_b^-),\\
W_2(z_1,0)=0,\q  z_1\in [0,N],\\
W_2(z_1,M)= \e \hat{W}_6^\#(M) f'(\hat{W}_6^\#(M)),\q  z_1\in [0,N],\\
W_4(N,z_2)=\e P_0(\hat{\theta}(N,z_2)),\q  z_2\in [0,M],
\end{cases}\end{equation}
where
$F_3(\hat{{\bf W}},\nabla\hat{{\bf W}},\bm{\Phi}^--\bm{\Phi}^-_b)$
and
$F_4(\hat{{\bf W}},\nabla\hat{{\bf W}},\bm{\Phi}^--\bm{\Phi}^-_b)$ are quadratic with respect to $\bW$ and $\nabla \bW$.
Since the values $ \hat{W}_6^\#(z_1, M)$ and $\hat{\theta}(N,z_2)$ may exceed the interval $[r_b,r_2]$ and $[0,\theta_0+\e f(r_2)]$, respectively,
one can also extend the functions $f$ and $P_{0}$ smoothly to a larger interval as  in \eqref{extension} and \eqref{vandermon}. The straightforward computations show
 \[
 F_3(\hat{{\bf W}},\nabla\hat{{\bf W}},\bm{\Phi}^--\bm{\Phi}^-_b)(z_1,0)=0\quad\text{and}\quad \p_{z_2} F_4(\hat{{\bf W}},\nabla\hat{{\bf W}},\bm{\Phi}^--\bm{\Phi}^-_b)(z_1,0)=0.
  \]
To obtain the estimate for $F_3$ and $F_4$, we should be careful about the singular terms involving sine and cotangent functions of $\hat{\th}(z)$ and $\theta_b(z_2)$. Note that there exists $\ka_i (i=1, 2)$ depending only on the background solutions such that
\[
\ka_1 z_2\leq \hat{\theta}(z)\leq \ka_2 z_2 \quad \text{for any}\,\, z\in \ol{E_+}. S
\]
Since $\hat{W}_2(z_1,0)=\hat{W}_3(z_1,0)=0$, it is easy to see that
\be\lab{elliptic231}
\sum_{j=2}^3\|\hat{W}_j^2\cot\hat{\theta}(z)\|_{1,\al;E_+}^{(1-\al;\Ga_{w,z})}\leq C \||\hat{{\bf W}}|\|^2.
\ee
Also by \eqref{theta21} and \eqref{background-theta}, one has
\be\no
\begin{aligned}
&\cos\hat{\theta}(z)-\cos\theta_b(z_2)=\f{1}{(r_b+z_1)^2\ti{\rho}^+_b(z_1)\ti{U}_b^+(z_1)} z_2^2-\hat{\vartheta}(z_1, z_2)
\end{aligned}
\ee
and
\begin{equation}
\begin{aligned}
&(\cot\hat{\theta}(z)- \cot\theta_b(z_2))\hat{W}_2(z)
= \f{\hat{W}_2(z)}{\sin\theta_b(z_2)}(\cos\hat{\theta}(z)-\cos\theta_b(z_2))\\
&\qquad \qquad+ \f{\cos\hat{\theta}(\cos\hat{\theta}(z)+\cos\theta_b(z_2))}{\sin\hat{\theta}(z)
+\sin\theta_b(z_2)}\f{\cos\hat{\theta}(z)-\cos\theta_b(z_2)}{\sin\hat{\theta}(z)\sin\theta_b(z_2)} \hat{W}_2(z).
\end{aligned}
\end{equation}
 With the aid of \eqref{elliptic231}, one has
\be\lab{elliptic233}
\sum_{j=3}^4\|F_i(\hat{{\bf W}},\nabla\hat{{\bf W}},\bm{\Phi}^--\bm{\Phi}^-_b)\|_{1,\al;E_+}^{(1-\al;\Ga_{w,z})}\leq C (\e+\||\hat{{\bf W}}|\|^2).
\ee

The crucial part for the analysis is to get the existence of solutions for the problem \eqref{elliptic23}.
Set
\be\no
\la_1(z_1)&=& \text{exp}\b(-\int_0^{z_1} \f{c^2(\ti{\rho}_b^+,S_b^+)+(\ti{U}_b^+)^2}{(r_b+z_1)(c^2(\ti{\rho}_b^+,S_b^+)-(\ti{U}_b^+)^2)} ds\b),\\ \no
\la_2(z_1)&=& \f{r_b+z_1}{\ti{U}_b^+(z_1)} \la_1(z_1),\ \la_3(z_1)= a\f{r_b+z_1}{\ti{U}_b^+(z_1)} \f{(z_1-N)\p_{z_1}\ti{P}_b^+}{N}  \la_1(z_1),\\ \no
\la_4(z_1)&=& \text{exp}\b(\int_0^{z_1} e_3(s) ds\b),\ \la_5(z_1)= \f{\ga \ti{P}_b^+ (\ti{U}_b^+)^2}{\ka_b(r_b+z_1) (c^2(\ti{\rho}_b^+,S_b^+)-(\ti{U}_b^+)^2)} \la_4(z_1),\\\no
\la_6(z_1)&=& \b(e_6(z_1)+e_2 e_4(z_1)\b) \la_4(z_1).
\ee
It is clear that
\begin{equation}
\la_1, \la_2, \la_4>0\quad \text{and}\quad \la_3\leq 0.
\end{equation}
In terms of $\lambda_i$ ($i=1, \cdots, 6$), the problem  \eqref{elliptic23} can be rewritten as
\be\lab{elliptic24}\begin{cases}
\p_{z_1}(\la_1(z_1) W_2) + \f{\sin\theta_b(z_2)}{2z_2}\p_{z_2} (\la_2(z_1) W_4) + \la_3(z_1) W_2(0,z_2)= G_1(z),\\
\p_{z_1}(\la_4(z_1) W_4) -\la_5(z_1) \f{\sin\theta_b(z_2)}{2z_2} (\p_{z_2} W_2 +\f{2\ka_b z_2\cos\theta_b(z_2)}{\sin^2\theta_b(z_2)} W_2 )\\
\q\q+\la_6(z_1)\b(W_6(M)-a \int_{z_2}^M \f{2s}{\sin\theta_b(s)} W_2(0,s) ds\b)= G_2(z),\\
W_4(0,z_2)= e_1 a\b(\f{W_6(M)}{a}- \int_{z_2}^M \f{2s}{\sin\theta_b(s)} W_2(0,s) ds\b)+G_3(z_2),\\
W_4(N,z_2)=\e G_4(z_2) ,\\
W_2(z_1,0)=0,\q W_2(z_1,M)=\e G_5(z_1),
\end{cases}\ee
where $a$ is given in \eqref{defa} and
\be\no
&&G_1(z) =\la_1(z_1) F_3(\hat{{\bf W}},\nabla\hat{{\bf W}},\bm{\Phi}^--\bm{\Phi}^-_b),\ \ G_2(z)= \la_4(z_1) F_4(\hat{{\bf W}},\nabla\hat{{\bf W}},\bm{\Phi}^--\bm{\Phi}^-_b),\\\no
&&G_3(z_2)=e_1 R_{12}(\hat{{\bf W}}(0,z_2),\bm{\Phi}^--\bm{\Phi}^-_b) +R_5(\hat{{\bf W}}(0,z_2)),\ \ G_4(z_2)=P_0(\hat{\theta}(N,z_2)), \\\no
&&G_5(z_1)=  \hat{W}_6^\#(z_1, M) f'(\hat{W}_6^\#(z_1, M)).
\ee
Note that
the first equation in \eqref{elliptic24} can be written  as follows
\be\no
\begin{aligned}
\p_{z_1}\b(\f{2z_2}{\sin\theta_b(z_2)}\la_1(z_1) W_2\b) + \p_{z_2}\b\{ &\la_2(z_1) W_4 + \la_3(z_1)\b(\f{W_6(M)}{a}-\int_{z_2}^{M} \f{2s}{\sin\theta_b(s)} W_2(0,s)ds\b)\\\no
&\q-\int_{z_2}^M G_1(z_1,s) ds\b\}=0.
\end{aligned}
\ee
Hence there exists  a potential function $\phi$ satisfying
\begin{equation}\lab{potential}
\left\{
\begin{aligned}
&\p_{z_1}\phi =\la_2(z_1) W_4 + \la_3(z_1)\b(\f{W_6(M)}{a}-\int_{z_2}^{M}\f{2s}{\sin\theta_b(s)} W_2(0,s)ds\b)-\int_{z_2}^M G_1(z_1,s) ds,\\
&\p_{z_2} \phi = -\la_1(z_1)\f{2z_2}{\sin\theta_b(z_2)} W_2(z),\q \phi(0,M)=0.
\end{aligned}
\right.
\end{equation}
Therefore, $W_2$ and $W_4$ can represented in terms of $\phi$ as follows
\be\lab{potential1}
\left\{
\begin{aligned}
W_2(z)&=-\f1{\la_1(z_1)}\f{\sin\theta_b(z_2)}{2z_2}\p_{z_2} \phi,\\
W_4(z)&= \f{\p_{z_1}\phi}{\la_2(z_1)}-\f{\la_3(z_1)}{\la_2(z_1)}\b(\f{W_6(M)}{a}-\phi(0,z_2)\b)+\f{1}{\la_2(z_1)} \int_{z_2}^M G_1(z_1,s)ds.
\end{aligned}
\right.
\ee
Now, substituting \eqref{potential1} into the second equation and the boundary conditions in \eqref{elliptic24} gives
\begin{equation}\lab{elliptic25}\begin{cases}
\p_{z_1}\b(\f{\la_4(z_1)}{\la_2(z_1)}\p_{z_1}\phi\b) - \b\{a\la_6(z_1)+\f{d}{dz_1}\b(\f{\la_4(z_1)\la_3(z_1)}{\la_2(z_1)}\b)\b\}(\phi(0,z_2)-\f{W_6(M)}{a})\b)\\
\q+ \f{\la_5(z_1)}{\la_1(z_1)} \b(\f{\sin\theta_b(z_2)}{2z_2}\p_{z_2}\b(\f{\sin\theta_b(z_2)}{2z_2}\p_{z_2}\phi\b)+\f{\ka_b\cos\theta_b(z_2)}{2z_2}\p_{z_2}\phi\b)\\
\q\q=\p_{z_2}\b(\int_0^{z_2} G_2(z_1,s)ds\b)- \p_{z_1}\b(\f{\la_4(z_1)}{\la_2(z_1)} \int_{z_2}^M G_1(z_1,s)ds\b),\\
\p_{z_1}\phi(0,z_2)+ (a \la_2(0)e_1+\la_3(0)) \b(\phi(0,z_2)-\f{W_6(M)}{a}\b)=\la_2(0)G_3(z_2)- \int_{z_2}^M G_1(0,s)ds,\\
\p_{z_1}\phi(N,z_2)= \epsilon \la_2(N) P_{0}(\hat{\theta}(N,z_2))- \int_{z_2}^M G_1(N,s)ds,\\
\p_{z_2}\phi(z_1,0) =0,\\
\p_{z_2}\phi(z_1,M)= -\f{2M}{\sin\theta_b(M)}\la_1(z_1)\e ( \hat{W}_6^\#(z_1, M))f'( \hat{W}_6^\#(z_1, M)).
\end{cases}\end{equation}

To simplify the notations, we define
\be\no
&&a_1(z_1)= \f{\la_4(z_1)}{\la_2(z_1)},\ a_2(z_1)=\f{\la_5(z_1)}{\la_1(z_1)}, \ a_3(z_1)=\b\{a\la_6(z_1)+\f{d}{dz_1}\b(\f{\la_4(z_1)\la_3(z_1)}{\la_2(z_1)}\b)\b\},\\\no
&&a_4=a e_1\la_2(0)+\la_3(0),\  \mu=-\f{W_6(M)}{a},\
\mcG_1(z_2)=\la_2(0)G_3(z_2)- \int_{z_2}^M G_1(0,s)ds\\\no
&&\mcF_1(z)=-\f{\la_4(z_1)}{\la_2(z_1)} \int_{z_2}^M G_1(z_1,s)ds,\ \ \  \mcF_2(z)= \int_0^{z_2} G_2(z_1,s)ds,\\\no
&& \mcG_2(z_2)= \e\la_2(N) P_{0}(\hat{\theta}(N,z_2))- \int_{z_2}^M G_1(N,s)ds,\ \  \mcG_3(z_1)=-\f{2M}{\sin\theta_b(M)} \la_1(z_1)G_5(z_1),\\\no
&& \mfd_1(z_2)= \f{\sin\theta_b(z_2)}{2z_2},\q \mfd_2(z_2)=\f{\ka_b \cos\theta_b(z_2)}{2z_2}.
\ee
It follows from \eqref{elliptic233} that
\be\lab{elliptic234}
\sum_{i=1}^2\|\mcF_i\|_{1,\al;E_+}^{(-\al;\Ga_{w,z})} +\sum_{i=1}^2 \|\mcG_i\|_{1,\al;(0,M)}^{(-\al;\{M\})}\leq C (\e+\||\hat{{\bf W}}|\|^2).
\ee

%

To deal with the singularity near $z_2=0$, we define
\be\no
\zeta_1=z_1,\ \zeta_2 =z_2 \cos \tau,\ \zeta_3 =z_2 \sin\tau,\quad \text{for } z_1\in [0, N], \,\, z_2\in [0, M], \,\, \tau\in [0,2\pi].
\ee
and denote
\be\no
&&E_1=\{(\zeta_1,\zeta_2,\zeta_3): 0<\zeta_1<N,\ \zeta_2^2+\zeta_3^2\leq M^2\},\ \Ga_{w,\zeta}=[0,N]\times \{(\zeta_2,\zeta_3):\zeta_2^2+\zeta_3^2=M^2\},\\\no
&&E_2= \{(\zeta_2,\zeta_3): \zeta_2^2+\zeta_3^2\leq M^2\},\q \Gamma_\zeta'=\{(\zeta_2,\zeta_3): \zeta_2^2+\zeta_3^2=M^2\},\\\no
&& \Upsilon(\bm \zeta)= \phi(\zeta_1,\sqrt{\zeta_2^2+\zeta_3^2}),
\ee
Denote $\Upsilon^*(\bm \zeta)= \Upsilon(\bm \zeta)+\mu$ where $\bm\zeta=(\zeta_1, \zeta_2, \zeta_3)$.
Then $\Upsilon^*$ satisfies the following problem
\be\lab{elliptic33}
\left\{
\begin{aligned}
&\p_{\zeta_1} (a_1(\zeta_1) \p_{\zeta_1} \Upsilon^*)  -\f{\ka_b^2}{4} a_2(\zeta_1) (\zeta_2\p_{\zeta_2}\Upsilon^*+ \zeta_3\p_{\zeta_3} \Upsilon^*)+ a_3(\zeta_1) \Upsilon^*(0,\zeta_2,\zeta_3)\\
&\q +a_2(\zeta_1) \mfd_1(\sqrt{\zeta_2^2+\zeta_3^2})\b[\p_{\zeta_2}(\mfd_1(\sqrt{\zeta_2^2+\zeta_3^2})\p_{\zeta_2} \Upsilon^*)+ \p_{\zeta_3}(\mfd_1(\sqrt{\zeta_2^2+\zeta_3^2})\p_{\zeta_3}\Upsilon^*)\b]\\
 &=  \p_{\zeta_1} \mcF_1(\zeta_1,\sqrt{\zeta_2^2+\zeta_3^2})+ \displaystyle\sum_{i=2}^3\p_{\zeta_i}\b(\f{\zeta_i \mcF_2(\zeta_1,\sqrt{\zeta_2^2+\zeta_3^2})}{\sqrt{\zeta_2^2+\zeta_3^2}}\b)-\f{\mcF_2(\zeta_1,\sqrt{\zeta_2^2+\zeta_3^2})}{\sqrt{\zeta_2^2+\zeta_3^2}},\\
&\p_{\zeta_1} \Upsilon^*(0,\zeta_2,\zeta_3)+ a_4 \Upsilon^*(0,\zeta_2,\zeta_3)= \mcG_1(\sqrt{\zeta_2^2+\zeta_3^2}),\\
&\p_{\zeta_1} \Upsilon^*(N,\zeta_2,\zeta_3)= \mcG_2(\sqrt{\zeta_2^2+\zeta_3^2}),\\
&(\zeta_2\p_{\zeta_2}+ \zeta_3\p_{\zeta_3}) \Upsilon^*(\zeta_1,\zeta_2,\zeta_3)= M \mcG_3(\zeta_1),\ \ \text{on}\ \zeta_2^2+\zeta_3^2=M^2.
\end{aligned}
\right.
\ee

\bp\lab{elliptic35}
{\it For any $(\mcF_1,\mcF_2)\in C_{1,\al;E_1}^{(-\al;\Ga_{w,\rm\zeta})}$ and $\mcF_2(x_1,0)=0$, $\mcG_1$, $\mcG_2\in C_{1,\al;E_2}^{(-\al;\Gamma_\zeta')}$, then the problem \eqref{elliptic33} has a unique solution $\Upsilon^*(\bm \zeta)= \tilde\Upsilon^*(\zeta_1,\sqrt{\zeta_2^2+\zeta_3^2})\in C_{2,\al;E_1}^{(-1-\al;\Ga_{w, \zeta})}$, which satisfies the following estimate
\be\lab{elliptic36}
\|\Upsilon^*\|_{2,\al;E_1}^{(-1-\al;\Ga_{w,\rm\zeta})}\leq C \left(\sum_{i=1}^2\|\mcF_i\|_{1,\al;E_1}^{(-\al;\Ga_{w,\rm\zeta})}+ \sum_{j=1}^2\|\mcG_j\|_{1,\al;E_2}^{(-\al;\Gamma_\zeta')}+\|\mcG_3\|_{1,\al;[0,N]}\right).
\ee
}\ep

\bpf
Note that the coefficients in the first equation of \eqref{elliptic33} are infinitely smooth near the axis $\zeta_2^2+\zeta_3^2=0$,
which is quite different from the elliptic system in  \cite[Lemma 4.3]{lxy10b}. So we do not need to take much care of the regularity near the axis.
This advantage  comes essentially from our new Lagrangian transformation. The system \eqref{elliptic33} has a variational structure similar
to the one in the proof of  \cite[Lemma 4.3]{lxy10b}, one can obtain the existence and uniqueness of $H^1(E_1)$ weak solution by Lax-Milgram
theorem and Fredholm alternative theorem as in \cite{lxy10b}.

To get the estimate \eqref{elliptic36}, one can put the term  $a_3(\zeta_1) \Upsilon^*(0,\zeta_2,\zeta_3)$
on the right hand side, so by the trace theorem, the right hand side belongs to $L^2(E_1)$ and the interior estimates can be obtained by a standard way.
Furthermore, one can use \cite[Theorems 5.36 and 5.45]{lieberman13} to obtain global $L^{\infty}$ bound and $C^{\alpha}$ norm estimates for $\Upsilon^*$ with
some H\"{o}lder exponent $\al\in (0,1)$. Hence the nonlocal term $a_3(\zeta_1) \Upsilon^*(0,\zeta_2,\zeta_3)$ becomes $C^{\al}$ and \eqref{elliptic36}
follows by employing  \cite[Theorem 4.6]{lieberman13}.
\epf

Proposition \ref{elliptic35} actually implies the following estimates for $W_2$ and $W_4$.
\bp\lab{solvability}
{\it The probelm \eqref{elliptic24} has a unique solution $(W_2, W_4, W_6(M))\in (C_{2,\al;E_+}^{(-\al;\Ga_{w,z})})^2\times \mbR$ satisfying
\be\lab{elliptic241}
\|W_2\|_{2,\al,E_+}^{(-\al;\Ga_{w,z})}+ \|W_4\|_{2,\al,E_+}^{(-\al;\Ga_{w,z})} + |W_6(M)|\leq C(\de^2 +\e)
\ee
and
\be\lab{compatibility241}
W_2(z_1,0)=\p_{z_2}^2 W_2(z_1,0)=0,\ \p_{z_2} W_4(z_1,0)=0.
\ee
}\ep

\bpf
It follows from Proposition \ref{elliptic35} and the equivalence between $\|\cdot\|_{1,\al;E_+}^{(-\al;\Ga_{w,z})}$ and $\|\cdot\|_{1,\al;E_1}^{(-\al;\Ga_{w,\zeta})}$ that the system \eqref{elliptic24} has a unique solution $(W_2,W_4,W_6(M))\in (C_{1,\al;E_+}^{(-\al;\Ga_{z,w})})^2\times \mbR$ satisfying
\be\lab{ellpitc242}
\begin{aligned}
&\|W_2\|_{1,\al, E_+}^{(-\al;\Ga_{w,z})} +\|W_4\|_{1,\al, E_+}^{(-\al;\Ga_{w,z})}+ |W_6(M)|\\
\leq& C(\sum_{i=1}^2 \|G_i\|_{1,\al;E_+}^{(1-\al;\Ga_{w,z})}+\|G_3\|_{1,\al;E_+}^{(-\al;\Ga_{w,z})}+ \e)\\\no
\leq& C(\||\hat{{\bf W}}|\|^2+\e)\leq C(\de^2+\e).
\end{aligned}
\ee
In addition,  $W_2(z_1,0)=\p_{z_2} W_4(z_1,0)=0$.

 Rewrite the problem \eqref{elliptic24} as
\be\lab{elliptic37}\begin{cases}
\p_{z_1}(\la_1(z_1) W_2) + \p_{z_2} (\la_2(z_1) W_4)= G_5(z),\\
\p_{z_1}(\la_4(z_1) W_4) -\la_5(z_1)\f{\sin\theta_b(z_2)}{2z_2} (\p_{z_2} W_2 +\f{2\ka_bz_2 \cos\theta_b(z_2)}{\sin^2\theta_b(z_2)} W_2)= G_6(z),\\
W_4(0,z_2)= G_8(z_2),\,\, W_4(N,z_2)=\e G_4(z_2) ,\\
W_2(z_1,0)=0,\,\, W_2(z_1,M)=\e G_5(z_1),
\end{cases}\ee
where
\be\no
\begin{aligned}
&G_5(z)= G_1(z)- \la_3(z_1) W_2(0,z_2),\\
& G_6(z)=G_2(z)+\la_6(z_1)\b(W_6(M)-a \int_{z_2}^M\f{2s}{\sin\theta_b(s)} W_2(0,s) ds\b),\\\no
&G_7(z)= e_1 a\b(\f{W_6(M)}{a}- \int_{z_2}^M \f{2s}{\sin\theta_b(s)} W_2(0,s) ds\b)+G_3(z_2).
\end{aligned}
\ee
Hence $W_4$ satisfies
\be\lab{elliptic38}\begin{cases}
\p_{z_1}\b(\f{2z_2}{\sin\theta_b(z_2)}\f{\la_1(z_1)}{\la_5(z_1)}\p_{z_1}(\la_4(z_1) W_4)\b)+ \la_2(z_1)\b(\p_{z_2}^2 W_4 + \f{2\ka_b z_2\cos\theta_b(z_2)}{\sin^2\theta_b(z_2)}\p_{z_2}W_4\b)\\
\q\q=\p_{z_1}\b(\f{2z_2}{\sin\theta_b(z_2)}\f{\la_1(z_1)}{\la_5(z_1)} G_6(z)\b)+\p_{z_2} G_5(z)+\f{2\ka_b z_2\cos\theta_b (z_2)}{\sin^2\theta_b(z_2)} G_5(z),\\
W_4(0,z_2)= G_7(z_2),\quad W_4(N,z_2)=\e G_4(z_2),\quad
\p_{z_2} W_4(z_1,0)=0.
\end{cases}\ee
 Similar to the proof of Proposition \ref{elliptic35}, one has
\begin{equation}\lab{elliptic39}
\begin{aligned}
\|W_4\|_{2,\al;E_+}^{(-\al;\Ga_{w,z})}\leq& C\b(\sum_{i=5}^6 \|G_i\|_{1,\al;E_+}^{(1-\al;\Ga_{z,w})} + \|G_7\|_{1,\al;(0,M)}^{(-\al;\{M\})}+ \e\b) \\
\leq& C(\||\hat{{\bf W}}|\|^2+\e)\leq C(\de^2+\e).
\end{aligned}
\end{equation}
This, together with the first equation in \eqref{elliptic37}, gives
\[
\|(\p_{z_1}^2 W_2,\p_{z_1z_2}^2 W_2)\|_{\al;E_+}^{(2-\al;\Ga_{w,z})}\leq C (\|W_4\|_{2,\al;E_+}^{(-\al;\Ga_{w,z})} + \|W_2\|_{1,\al, E_+}^{(-\al;\Ga_{w,z})}) \leq C (\delta^2 +\epsilon).
\]
Finally, note that
\be\no
W_2(z)= \f{2}{\la_5(z_1)\sin \theta_b(z_2)} \int_0^{z_2}s (\p_{z_1}(\la_4(z_1) W_4)(z_1,s)-G_6(z_1,s)) ds.
\ee
Similar to \cite[Lemma B.3]{lxy10b}, we conclude that $W_2$ satisfies \eqref{elliptic241} and $\p_{z_2}^2 W_2(z_1,0)=0$.
\epf

\subsection{The iteration scheme for $W_1$ and the estimate for $W_1$, $W_3$, $W_5$, and $W_6$}
It follows from \eqref{bernoulli12} that  $W_1$ can be solved as follows
\be\lab{bernoulli32}
\begin{aligned}
W_1=&\f{1}{\ti{U}_b^+}\{B^-(\hat{W}_6^\diamondsuit(z_2),z_2)-B_b^- -[h(\ti{P}_b^+ + W_4, S_b^+ +W_5)-h(\ti{P}_b^+, S_b^+)]\}\\
&\q\q-\f{1}{2 \ti{U}_b^+}[\hat{W}_1^2 + (\ti{U}_b^++ \hat{W}_1)^2 \hat{W}_2^2+\hat{W}_3^2].
\end{aligned}
\ee
Now we are ready to estimate $W_1$, $W_3$, $W_5$, and $W_6$.
\begin{proposition}\label{other}
{\it With $(W_2,W_4)\in \left(H_{2,\alpha;E_+}^{(-\alpha;\Gamma_{w,z})}\right)^2$ obtained in Proposition \ref{solvability}, $W_6$, $W_5$, $W_3$, and $W_1$ are uniquely determined by \eqref{shock23}, \eqref{entropy21}, \eqref{swirl22} and \eqref{bernoulli32} and satisfy
\begin{equation}\label{other1}
\sum_{j=1,3,5}\|W_j|_{2,\al;E_+}^{(-\al;\Ga_{w,z})}+\|W_6\|_{3,\al,[0,M)}^{(-1-\al;\{M\})}
\leq C(\de^2+\e).
\end{equation}
}\end{proposition}
\begin{proof}
It follows from \eqref{shock23} that
\be\lab{shock31}
\begin{aligned}
W_6(z_2)=& W_6(M)- a\int_{z_2}^M \f{2s}{\sin\theta_b(s)} W_2(0,s) ds\\
&\q\q-\int_{z_2}^M R_{11}(\hat{{\bf W}}(0,s), \bm{\Phi}^-(r_b+ \hat{W}_6(s),s)- \bm{\Phi}_b^-(r_b+\hat{W}_6(s)))ds.
\end{aligned}
\ee
Thus $W_6'(0)=0$ and the following estimate holds
\be\lab{shock32}
\begin{aligned}
\|W_6\|_{3,\al,[0,M)}^{(-1-\al;\{M\})} \leq& C(|W_6(M)|+\|W_2\|_{2,\al,E_+}^{(-\al;\Ga_{w,z})}+\|R_{11}(\hat{W},\bm{\Phi}^- -\bm{\Phi}^-_b)\|_{2,\al,E_+}^{(-\al;\Ga_{w,z})})\\
\leq & C(\de^2+\e).
\end{aligned}
\ee
It follows from \eqref{entropy21} that
\be\lab{entropy31}
W_5(z)=W_5(0,z_2)= e_2 W_6(z_2) + R_4(\hat{{\bf W}},\bm{\Phi}^- -\bm{\Phi}^-_b).
\ee
Hence $\p_{z_2} W_5(z_1,0)=0$ and
\be\lab{entropy32}
\begin{aligned}
\|W_5\|_{2,\al,E_+}^{(-\al;\Ga_{w,z})}\leq& e_2 \|W_6\|_{3,\al,[0,M)}^{(-1-\al;\{M\})}+\|R_4\|_{2,\al,E_+}^{(-\al;\Ga_{w,z})}
\leq  C(\de^2+\e).
\end{aligned}
\ee

Using \eqref{swirl22} gives
\begin{eqnarray}\label{swirl32}
\|W_3\|_{2,\al;E_+}^{(-\al;\Ga_{w,z})}&\leq& C\|\hat{{\bf W}}\|_{\Xi_{\de}}\|U_3^-\|_{C^{2,\alpha}(\Omega)}\leq C\epsilon\delta.
\end{eqnarray}

It follows from \eqref{bernoulli32} that
\be\lab{bernoulli33}
\|W_1\|_{2,\al;E_+}^{(-\al;\Ga_{w,z})} &\leq& C\left(\e+ \sum_{j=3}^4\|W_i\|_{2,\al;E_+}^{(-\al;\Ga_{w,z})} + \||\hat{{\bf W}}|\|^2\right)\leq C(\e+\de^2).
\ee
Combining \eqref{shock32} with \eqref{entropy32}-\eqref{bernoulli33} together finishes the proof of the proposition.
\end{proof}

\subsection{Proof of Theorem \ref{transonic}} Now we are in position to prove Theorem \ref{transonic}.
\bpf[Proof of Theorem \ref{transonic}.] The proof is divided into three steps.

{\it Step 1. Boundedness.} Given any $\hat{{\bf W}}\in \Xi_{\de}$, let ${\bf W}=\mc{T}(\hat{{\bf W}})$ be the solutions obtained in Propositions \ref{solvability} and \ref{other}. Thus one has
\be\lab{final}
\||{\bf W}|\|\leq C_*(\e+\de^2).
\ee
Let $\de=2 C_*\e$ and choose $\e_0$ small enough satisfying $2C_*^2\e_0\leq \f12$. Therefore, for any $0<\e\leq \e_0$, one has
\be\no
C_*(\e+\de^2)=\f{\de}{2}+ 2C_*^2\e\de \leq \f{\de}{2}+\f{\de}{2}=\de.
\ee
This implies that $\mc{T}$ maps $\Xi_\de$ into itself.

{\it Step 2. Contraction.} Given any $\hat{{\bf W}}^{(i)}\in \Xi_\de$ ( $i=1$, $2$), let ${\bf W}^{(i)}=\mc{T} \hat{{\bf W}}^{(i)}$ ($i=1$, $2$) be obtained in Step 1. Denote
\be\no
\hat{{\bf Y}}= \hat{{\bf W}}^{(1)}- \hat{{\bf W}}^{(2)}\quad \text{and}\quad  {\bf Y}= {\bf W}^{(1)}-{\bf W}^{(2)}.
\ee

 It follows from \eqref{elliptic23} that $Y_2$ and $Y_4$ satisfies
\be\lab{elliptic41}\begin{cases}
\p_{z_1}(\la_1(z_1) Y_2) +\f{\sin\theta_b(z_2)}{2z_2} \p_{z_2} (\la_2(z_2) Y_4) + \la_3 Y_2(0,z_2)= G_1^{(1)}(z)-G_1^{(2)}(z),\\
\p_{z_1}(\la_4(z_1) Y_4) -\la_5(z_1)\f{\sin\theta_b(z_2)}{2z_2} (\p_{z_2} Y_2 +  \f{2\ka_bz_2\cos\theta_b(z_2)}{\sin^2\theta_b(z_2)}Y_2)\\
\q\q- \la_6(z_1)\b(Y_6(M)-a \int_{z_2}^M \f{2s}{\sin\theta_b(s)} Y_2(0,s) ds\b)= G_2^{(1)}(z)-G_2^{(2)}(z),\\
Y_4(0,z_2)= e_1 a\b(\f{Y_6(M)}{a}- \int_{z_2}^M \f{2s}{\sin\theta_b(s)} Y_2(0,s) ds\b)+G_3^{(1)}(z_2)-G_3^{(2)}(z_2),\\
Y_4(N,z_2)=G_4^{(1)}(z_2)-G_4^{(2)}(z_2),\\
Y_2(z_1,0)=0,\q Y_2(z_1,M)= G_5^{(1)}(z_1)-G_5^{(2)}(z_1).
\end{cases}\ee
Using Proposition \ref{solvability} gives
\begin{equation}\lab{elliptic42}
\begin{aligned}
\sum_{i=2,4}\|Y_i\|_{2,\al;E_+}^{(-\al;\Ga_{w,z})}+|Y_6(M)|\leq& C\sum_{i=1}^2 \|G_i^{(1)}-G_i^{(2)}\|_{1,\al;E_+}^{(1-\al;\Ga_{w,z})}+ \|G_3^{(1)}-G_3^{(2)}\|_{1,\al;[0,M)}^{(-\al;\{M\})}\\
&\q\q+ \e \|P_0(\hat{\theta}^{(1)})-P_0(\hat{\theta}^{(2)})\|_{1,\al;E_+}^{(-\al;\{M\})} +C\e |\hat{Y}(M)|\\
\leq& C\e \left(\sum_{i=1}^5 \|\hat{Y}_i\|_{2,\al;E_+}^{(-\al;\Ga_{w,z})}+\|\hat{Y}_6\|_{3,\al;[0,M)}^{(-1-\al;\{M\})}\right).
\end{aligned}
\end{equation}

It follows from \eqref{shock31} that $Y_6$ satisfies
\be\lab{shock41}
Y_6(z_2)= Y_6(M)- \int_{z_2}^M \f{2s}{\sin\theta_b(s)} Y_2(0,s)ds + R_{12}^{(1)}-R_{12}^{(2)}.
\ee
Therefore, one has
\be\lab{shock42}
\|Y_6\|_{3,\al; [0,M)}^{(-1-\al;\{M\})}&\leq& |Y_6(M)| +C \|Y_2\|_{2,\al; E_+}^{(-\al;\Ga_{w,z})}+ \|R_{11}^{(1)}-R_{11}^{(2)}\|_{2,\al; E_+}^{(-\al;\Ga_{w,z})}\\\no
&\leq& C\e \||\hat{{\bf Y}}|\|.
\ee

It follows from \eqref{entropy31} that
\be\lab{entropy41}
Y_5(z)= e_2 Y_6(z_2) + R_4^{(1)}- R_4^{(2)}.
\ee
Thus it holds that
\be\lab{entropy42}
\|Y_5\|_{2,\al;E_+}^{(-\al;\Ga_{w,z})} &\leq& C\|Y_6\|_{3,\al; [0,M)}^{(-1-\al;\{M\})}+ \|R_4^{(1)}-R_4^{(2)}\|_{2,\al;E_+}^{(-\al;\Ga_{w,z})}\leq C\e \||\hat{{\bf Y}}|\|.
\ee

The equation \eqref{swirl22} implies
\be\lab{swirl40_5}
Y_3(z_1,z_2)&=&\f{\hat{W}_6^\diamondsuit(z_2)}{\hat{W}_6^\#(z_1,z_2)} \f{\sin\hat{\theta}(0,z_2)}{\sin\hat{\theta}(z_1,z_2)} U_3^-(\hat{W}_6^\diamondsuit(z_2),z_2).
\ee
Thus one has
\be\lab{swirl41}
\|Y_3\|_{2,\al;E_+}^{(-\al;\Ga_{w,z})}\leq C\e \||\hat{{\bf Y}}|\|.
\ee

Finally, \eqref{bernoulli32} implies that
\be\lab{bernoulli41}
\begin{aligned}
\|Y_1\|_{2,\al;E_+}^{(-\al;\Ga_{w,z})}&\leq C(\e \|\hat{Y}_6\|_{3,\al;(0,M)}^{(-1-\al;\{M\})}+\sum_{j=3}^4\|Y_j\|_{2,\al;E_+}^{(-\al;\Ga_{w,z})}+C\e \||\hat{{\bf Y}}|\|\\
&\leq C\e\||\hat{{\bf Y}}|\|.
\end{aligned}
\ee

Collecting all the estimates \eqref{elliptic42}, \eqref{shock42}, \eqref{entropy42}, \eqref{swirl41}, and \eqref{bernoulli41} together gives
\be\lab{total}
\||{\bf Y}|\|\leq C_{\sharp} \e \||\hat{{\bf Y}}|\|.
\ee
Obviously, if one chooses $\e_0\leq \min\{\f1{4C_*^2}, \f{1}{2C_{\sharp}}\}$, then $\mc{T}$ is a contraction mapping for $\Xi_\de$ to $\Xi_\de$. Hence $\mc{T}$ must have a fixed point in $\Xi_\de$. It is easy to see that this fixed point is a solution for the problem \eqref{shock12}, \eqref{entropy12},  \eqref{swirl12}, \eqref{bernoulli12}, and \eqref{elliptic13}. Furthermore, since the Lagrangian transformation is invertible, the associated solution  $(U_1^+,U_2^+,U_3^+, P^+,S^+)$ and $\xi$ satisfy the properties listed in \eqref{transonic01} and \eqref{subsonic01}.

{\it Step 3. Uniqueness.}  Suppose that there are two solutions $(U_{1}^{+,(j)},U_{2}^{+,(j)},U_{3}^{+,(j)}, P^{+,(j)},S^{+,(j)})$ and $\xi_j$ ($j=1$, $2$) satisfying the properties \eqref{transonic01} and \eqref{subsonic01}. We can perform the corresponding Lagrangian transformation and decompose the Euler system as above, in this case we do not need to use the extension \eqref{extension} any more because  the existence of solutions has been assumed. It is the same as the proof for that the operator $\mc{T}$ is a contraction mapping. Therefore, these two solutions are indeed the same.
\epf

\section{High order regularity of the transonic shock solution}\noindent


In this section, we show that the regularity of the shock front and subsonic solutions can be improved if  the nozzle wall is not perturbed and  the supersonic incoming flow satisfies some additional compatibility conditions.

In the following lemma, we show that the compatibility conditions \eqref{super3} and \eqref{pressure1} for the supersonic solutions are preserved along the straight wall.
\bl\lab{supersonic}  {\it If \eqref{super3} and \eqref{pressure1} hold, the system \eqref{euler-sph1} supplemented with \eqref{super1} and \eqref{wallslip} has a unique
smooth solution $\bPsi^-=(U_1^-,U_2^-,U_3^-,P^-,S^-)(r,\theta)\in
C^{2,\al}(\bar{\Om}).$ Moreover, this solution $\bPsi^-$ satisfies
\begin{equation}\lab{2114}
\|(U_1^-,U_2^-,U_3^-,P^-,S^-)-({U}_{0}^-,0,0,\hat{P}_0^-, \hat{S}_0^-)\|_{C^{2,\al}(\overline{\Om})}\leq C_0\e,
\end{equation}
where the positive constant $C_0$ depends only on $\al$ and the supersonic
incoming flow.

If, in addition, $\Psi_{en}^-$ satisfies \eqref{compati2}, then the solutions $\bm\Psi^-$ satisfies
\be
 \f{\p}{\p \th}(U_1^-,U_3^-, P^-,S^-)(r, \th_0)=0.
\ee
}
\el
\bpf
Since $U_2(r,\theta_0)\equiv 0$, it follows from the third, fourth and fifth equation of \eqref{euler-sph1} that one has
\be\lab{541}
\p_{\theta} P- (\rho U_3^2) \cot \theta=0,\q  (r\p_r U_3+ U_3)=0,\q \p_r S=0 \quad \text{for }\theta=\theta_0.
\ee
Furthermore, differentiating the fifth equation of \eqref{euler-sph1} with respect to $\theta$ yields
\be\no
\rho U_1\p_r(\p_{\theta} S)(r,\theta_0) + \f{\rho}{r}\p_{\theta} U_2 \p_{\theta}S(r,\theta_0)=0.
\ee
Therefore, $\p_{\th} S(r,\theta_0)\equiv 0$ as long as $\p_{\theta}S(r_1,\theta_0)=0$.

If $U_3(r_1,\theta_0)\equiv 0$, then one can conclude  $U_3(r,\theta_0)\equiv 0$ from \eqref{541}. Using \eqref{541} again yields $\p_{\theta} P(r,\theta_0)=0$ and $\p_r U_3(r,\theta_0)\equiv 0$. Differentiating the second equation of \eqref{euler-sph1} with respect to $\theta$ gives
\be\no
\rho U_1\p_r(\p_{\th} U_1)(r,\theta_0) + \rho \p_r U_1 \p_{\theta} U_1 (r,\theta_0) + \f{\rho}{r} \p_{\theta} U_2 \p_{\theta} U_1(r,\theta_0)=0.
\ee
 Hence, $\p_{\th} U_1(r,\theta_0)\equiv 0$ provided  $\p_{\theta} U_1(r_0,\theta_0)=0$. The compatibility conditions at $\theta=0$ can be obtained similarly except for the second derivative $\p_{\theta}^2 U_2(r,0)=0$, which can be obtained by differentiating the first equation of \eqref{euler-sph1} with respect to $\theta$.
\epf

In the next lemma, we give the compatibility conditions of the subsonic flows at the intersection circles of the shock front and the nozzle wall as long as the assumptions of Lemma \ref{supersonic} hold.
\bl\lab{compatibility}
{\it  If the system \eqref{euler-sph1} with \eqref{super1},  \eqref{pressure},  \eqref{wallslip} and \eqref{compati2}, has a solution
$$(U_1^{\pm}(r,\th),U_2^{\pm}(r,\th),U_3^{\pm}(r,\theta), P^{\pm}(r,\th),S^{\pm}(r,\th))\in C^{2,\al}(\ol{\Om^{\pm}})$$
and $\xi(\th)\in C^{3,\al}([0,\theta_0])$, then the following compatibility
conditions on the nozzle wall and the symmetry axis hold
\begin{equation}\lab{231}
\left\{
\begin{aligned}
&\p_{\th}(U_1^+,U_3^+, P^+, S^+)(r,\th_0)\equiv 0,\quad \p_{\th}(U_1^+,U_3^+, P^+, S^+)(r,0)\equiv 0,\\
&U_2(r,0)^+=U_3^+(r,0)=U_2^+(r,\th_0)=U_3^+(r,\theta_0)=0,\,\, \p_{\th}^2U_2^+(r,0)=\p_{\th}^2U_2^+(r,\theta_0)=0,\\
&\xi'(0)=\xi'(\th_0)=0,\quad \xi^{(3)}(0)=0.
\end{aligned}
\right.
\end{equation}
}\el
\bpf It follows from the boundary condition \eqref{wallslip} and the jump
conditions \eqref{rh-sph} that
$$
U^+_2(r,0)=U_2^+(r,\th_0)=0, \q \xi'(0)=\xi'(\th_0)=0.
$$
Furthermore, the fourth equation in \eqref{rh-sph} implies that $U_3^+(\xi(\theta_0),\theta_0)= U_3^-(\xi(\theta_0),\theta_0)=0$. Thus it follows from  the fourth equation in \eqref{euler-sph1} that  $U_3^+(r,\theta_0)=0$ for any $r\in [\xi(\th_0),r_2]$. Therefore, $\f{\p}{\p\theta}P^+(r,\theta_0)\equiv 0$.

Differentiating the first, the second, the
fourth, and the fifth equations in \eqref{rh-sph} along the shock front gives
$$
\begin{cases}
\p_{\th}(\rho^+ U^+_1)(\xi(\theta_0)+, \theta_0)=\p_{\th}(\rho^- U_1^-)(\xi(\theta_0)-, \theta_0),\\
\p_{\th}(\rho^+ (U_1^+)^2+P^+)(\xi(\theta_0)+, \theta_0)=\p_{\th}(\rho^- (U_1^-)^2+P^-)(\xi(\theta_0)-, \theta_0),\\
\p_{\th} U_3^+(\xi(\theta_0)+, \theta_0) =\p_{\th} U_3^-(\xi(\theta_0)-, \theta_0),\\
\p_{\th}\left(e^+ +\f{|U^+|^2}{2}+\f{P^+}{\rho^+}\right)(\xi(\theta_0)+, \theta_0)=\p_{\th}\left( e^-+\f{|U^-|^2}{2}+\f{P^-}{\rho^-} \right)(\xi(\theta_0)-, \theta_0).
\end{cases}
$$
It follows from Lemma \ref{supersonic} that $\p_{\th}(U_1^-,U_3^-,P^-,S^-)(r,\th_0)=0$. These then imply that $\p_{\th} U_3^+(\xi(\th_0),\theta_0)=0$ and
$$
\begin{cases}
\p_{\th}(\rho^+ U_1^+)(\xi(\theta_0)+, \theta_0)=0,\\
(\rho U_1^+\p_{\th} U_1^+ +\p_{\th}P^+)(\xi(\theta_0)+, \theta_0)=0,\\
\p_{\th}(e^++\f{|U^+|^2}{2}+\f{P^+}{\rho^+})(\xi(\theta_0)+, \theta_0)=0,
\end{cases}
$$
which yields
\begin{eqnarray}\label{corner}
\p_{\th}U_1^+(\xi(\th_0),\th_0)=\p_{\th}S^+(\xi(\th_0),\th_0)=\p_{\th}\rho^+(\xi(\th_0),\th_0)=0.
\end{eqnarray}

Differentiating the second and the fifth equation in \eqref{rh-sph} with respect to $\theta$ yields
$$
\begin{cases}
\left\{U_1\p_{r}(\p_{\th}U_1^+)+(\p_{r}U_1^++\f{1}{r}\p_{\th}U_2^+)\p_{\th}U_1^++\f{U_1^+\p_r U_1^+\p_{S}\rho}{\rho} \p_{\th} S^+\right\}(r, \theta_0)=0,\\
\left\{U_1\p_{r}(\p_{\th}S^+)+\f{1}{r}\p_{\th}U_2^+\p_{\th}S^+ +\p_{r}S^+\p_{\th}U_1^+\right\}(r, \theta_0)=0.
\end{cases}
$$
This, together with \eqref{corner}, implies
\[
\p_{\th}U_1^+(r,\th_0)=\p_{\th}S^+(r,\th_0)=\p_{\th}\rho^+(r,\th_0)\equiv 0\quad \text{for}\,\, r\in (\xi(\theta_0), r_2].
\]
It follows from the equation for $U_3^+$ (the fourth equation in \eqref{euler-sph1}) that  one has
\be\no\begin{cases}
\left\{U_1^+\p_r(\p_{\th}U_3^+) + \f{U_1^+}{r}\p_{\th} U_3^+ + \f{\p_{\th} U_2^+}{r} \p_{\th} U_3^+\right\}(r, \theta_0)=0,\\
\p_{\th}U_3^+(\xi(\th_0),\th_0)=0.
\end{cases}\ee
Hence $\p_{\th} U_3^+(r,\th_0)\equiv 0$.

In addition, differentiating the first equation of \eqref{euler-sph1} with
respect to $\th$ leads to
$$
\p_{\th}^2U_2^+(r,0)=0.
$$
Furthermore, differentiating the third equation of
\eqref{rh-sph} along the shock front twice yields
$$
\xi^{(3)}(0)=0.
$$
Hence The proof of Lemma \ref{compatibility} is completed.
\epf

With the help of Lemmas \ref{supersonic} and \ref{compatibility}, one can prove  Theorem \ref{main2}.

\begin{proof}[Proof of Theorem \ref{main2}]
First, if the nozzle boundary is straight, then $\varpi$ and $P$ satisfy the following system
\be\lab{elliptic01}\begin{cases}
\p_{\theta}\varpi + \varpi\cot\theta -r\b(\f{1}{\rho U_1^2}-\f{1}{\rho c^2(\rho,S)}\b)\p_r P+\f{\varpi}{\rho c^2(\rho,S)}\p_{\theta} P+(\varpi^2+2)+\f{U_3^2}{U_1^2}=0,\\
\p_r \varpi -\f{\varpi}{r}-\f{\varpi^2}{r}\cot\theta+ \b(\f{1}{\rho U_1^2}-\f{\varpi^2}{\rho c^2(\rho,S)}\b)\f{1}{r}\p_{\theta} P -\f{\varpi}{\rho c^2(\rho,S)}\p_r P-\f{U_3^2}{r U_1^2}\cot\theta =0.
\end{cases}\ee
Comparing with \cite[equation (2.20)]{lxy10b}, both of the additional terms $\f{U_3^2}{U_1^2}$ and $\f{U_3^2}{r U_1^2}\cot\theta$ in \eqref{elliptic01} can be regarded as error terms and do not cause any trouble. Moreover, $U_3$ satisfies
\be\lab{swirl00}\begin{cases}
U_1\p_r (r U_3\sin\theta) + \f{U_2}{r}\p_{\theta} (r U_3\sin\theta)=0,\\
U_3(\xi(\theta),\theta)= U_3^-(\xi(\theta),\theta).
\end{cases}\ee
The transport equation \eqref{swirl00} can be uniquely solved by characteristic method. Furthermore, we can use the standard even extension (a simple modification for \cite[Lemma A]{xyy09} ) to get $C^{2,\alpha}(\Omega^+)$ regularity near the corner. The detailed proof of Theorem \ref{main2} is very similar to the proof for  \cite[Theorem 1.1]{lxy10b}, so we omit it here.
\end{proof}

{\bf Acknowledgement.}  Part of this work was done when Weng and Xie were visiting The Institute of Mathematical Sciences of The Chinese University of Hong Kong. They are grateful to the institute for providing nice research environment. Weng is partially supported by NSFC 11701431, the grant of One Thousand Youth Talents Plan of China (No. 212100004) and the Fundamental Research Funds for the Central Universities Grant 201-413000047. Xie is supported in part by NSFC grants 11971307 and 11631008, and Young Changjiang Scholars of  Ministry of Education. Xin is supported in part by the Zheng Ge Ru  Foundation, Hong Kong RGC Earmarked Research Grants CUHK-14305315, CUHK-14300917,  CUHK-14302917, and CUHK-14302819, and NSFC/RGC Joint Research Scheme Grant N-CUHK443/14.

\bibliographystyle{plain}

\end{document}